\documentclass[a4paper,reqno]{amsart}

\pdfoutput=1

\usepackage{amsmath,amssymb,amsthm,url,verbatim,enumitem,stmaryrd,mathtools,microtype}
\usepackage[utf8]{inputenc}
\usepackage[T1]{fontenc}
\usepackage{libertine}
\usepackage[libertine,cmintegrals,cmbraces]{newtxmath}
\usepackage[dvipsnames,svgnames,x11names]{xcolor}
\definecolor{darkred}{RGB}{139,0,0}
\definecolor{darkblue}{RGB}{0,0,139}
\definecolor{darkgreen}{RGB}{0,100,0}
\usepackage[pagebackref]{hyperref}
\hypersetup{
  colorlinks   = true, 
  urlcolor     = darkred, 
  linkcolor    = darkred, 
  citecolor   = darkgreen}
\usepackage{tikz}
\usepackage{tikz-cd}
\usetikzlibrary{decorations.markings}
\usepackage[capitalise,noabbrev]{cleveref}

\calclayout

\AtBeginDocument{%
   \def\MR#1{}
}

\setenumerate[1]{label=(\roman*),}
\setitemize[1]{label=\raisebox{0.25ex}{\tiny$\bullet$}}

\newcommand\smallsquare{{\mathbin{\text{\raise0.17ex\hbox{\scalebox{.7}{$\blacksquare$}}}}}}

\newcommand{\circled}[1]{\raisebox{.5pt}{\textcircled{\raisebox{-.9pt} {#1}}}}

\newcommand{\SO}{\ensuremath{\mathrm{SO}}}

\newcommand{\BO}{\ensuremath{\mathrm{BO}}}

\newcommand{\Fr}{\ensuremath{\mathrm{Fr}}}
\newcommand{\Diff}{{\ensuremath{\mathrm{Diff}}}}

\newcommand{\Map}{\ensuremath{\mathrm{Map}}}

\newcommand{\CE}{\ensuremath{\mathrm{CE}}}
\newcommand{\CbE}{\ensuremath{\smash{\mathrm{C\overline{E}}}}}

\newcommand{\E}{\ensuremath{\mathrm{E}}}
\newcommand{\I}{\ensuremath{\mathrm{I}}}
\newcommand{\CI}{\ensuremath{\mathrm{CI}}}
\newcommand{\CB}{\ensuremath{\mathrm{CB}}}
\newcommand{\CF}{\ensuremath{\mathrm{CF}}}
\newcommand{\Conf}{\ensuremath{\mathrm{Conf}}}
\newcommand{\BC}{\ensuremath{\mathrm{BC}}}
\newcommand{\sCE}{\ensuremath{\mathrm{sCE}}}

\newcommand{\C}{\ensuremath{\mathrm{C}}}
\newcommand{\interior}{\ensuremath{\mathrm{int}}}
\newcommand{\closure}{\ensuremath{\mathrm{closure}}}
\newcommand{\inc}{\ensuremath{\mathrm{inc}}}

\newcommand{\res}{\ensuremath{\mathrm{res}}}
\newcommand{\Inj}{\ensuremath{\mathrm{Inj}}}

\DeclareMathAlphabet{\mathpzc}{OT1}{pzc}{m}{it}
\newcommand{\catsingle}[1]{\ensuremath{\mathcal{#1}}}

\newcommand{\oO}{\ensuremath{\mathrm{O}}}

\newcommand{\bfH}{\ensuremath{\mathbf{H}}}

\newcommand{\bfR}{\ensuremath{\mathbf{R}}}
\newcommand{\bfZ}{\ensuremath{\mathbf{Z}}}
\newcommand{\bfQ}{\ensuremath{\mathbf{Q}}}

\newcommand{\cP}{\ensuremath{\catsingle{P}}}

\newcommand{\hofib}{\ensuremath{\mathrm{hofib}}}

\newcommand{\hocolim}{\ensuremath{\mathrm{hocolim}}}

\newcommand{\holim}{\ensuremath{\mathrm{holim}}}

\newcommand{\ra}{\rightarrow}
\newcommand{\lra}{\longrightarrow}

\newcommand{\xra}[1]{\xrightarrow{#1}}
\newcommand{\xlra}[1]{\overset{#1}{\longrightarrow}}
\newcommand{\xlla}[1]{\overset{#1}{\longleftarrow}}

\renewcommand{\Top}{\mathrm{Top}}

\newcommand{\im}{\mathrm{im}}

\newcommand{\id}{\mathrm{id}}

\newcommand{\Sect}{\mathrm{Sect}}
\newcommand{\pr}{\mathrm{pr}}

\newcommand{\TOP}{{\mathrm{Top}}}
\newcommand{\PL}{{\mathrm{PL}}}

\newtheorem{bigthm}{Theorem}
\newtheorem{bigcor}[bigthm]{Corollary}

\newtheorem{thm}{Theorem}[section]

\newtheorem{lem}[thm]{Lemma}
\newtheorem{prop}[thm]{Proposition}
\newtheorem{cor}[thm]{Corollary}

\theoremstyle{definition}

\newtheorem{convention}[thm]{Convention}

\theoremstyle{remark}
\newtheorem{ex}[thm]{Example}
\newtheorem*{nex}{Example}
\newtheorem{rem}[thm]{Remark}

\newtheorem*{history}{Historical remark}

\newcommand{\ul}[1]{\underline{#1}}

\begin{document}

\title{Stability of concordance embeddings}

\author{Thomas Goodwillie}
\address{Department of Mathematics, Brown University, 151 Thayer St., Providence RO 02912, USA}
\email{thomas\_goodwillie@brown.edu}

\author{Manuel Krannich}
\address{Department of Mathematics, Karlsruhe Institute of Technology, 76131 Karlsruhe, Germany}
\email{krannich@kit.edu}

\author{Alexander Kupers}
\address{Department of Computer and Mathematical Sciences, University of Toronto Scarborough, 1265 Military Trail, Toronto, ON M1C 1A4, Canada}
\email{a.kupers@utoronto.ca}

\begin{abstract}
We prove a stability theorem for spaces of smooth concordance embeddings. From it we derive various applications to spaces of concordance diffeomorphisms and homeomorphisms.
\end{abstract}

\maketitle

\tableofcontents

\section{Introduction}

Let $P\subset M$ be a compact submanifold of a smooth $d$-dimensional manifold $M$ such that $P$ meets $\partial M$ transversely. Writing $I\coloneq [0,1]$, a \emph{concordance embedding} of $P$ into $M$ is a smooth embedding $e\colon P \times I \hookrightarrow M \times I$ such that \begin{enumerate}
 	\item $e^{-1}(M \times \{i\})=P \times \{i\}$ for $i=0,1$ and
 	\item $e$ agrees with the inclusion on a neighbourhood of $P \times \{0\}\cup (\partial M\cap P)\times I\subset P\times I$.
\end{enumerate}
The space of such embeddings, equipped with the smooth topology, is denoted by $\CE(P,M)$. There is a stabilisation map \[\CE(P,M) \lra \CE(P \times J,M \times J)\] given by taking products with $J\coloneq [-1,1]$ followed by bending the result appropriately to make it satisfy the boundary condition (compare \cref{fig:stabilisation-intro}). In this work we establish a connectivity estimate for this map based on the disjunction results of \cite{Goodwillie}. To state it, recall that the \emph{handle dimension} of the inclusion $\partial M\cap P\subset P$  is the smallest number $p$ such that $P$ can be built from a closed collar on $\partial M\cap P$ by attaching handles of index at most $p$. 

\begin{bigthm}\label{thm:main}If the handle dimension $p$ of $\partial M\cap P\subset P$ satisfies $p\le d-3$, then the map
\[
	\CE(P,M) \lra \CE(P \times J,M \times J)
\]
is $(2d-p-5)$-connected.\end{bigthm}

\begin{rem}\label{EandCE}\ 
\begin{enumerate}
\item We prove \cref{thm:main} as the case $r=0$ of a stronger multirelative stability theorem about a map of $r$-cubes of spaces of concordance embeddings (see \cref{thm:main-relative}).
\item \label{EandCE:ii}The individual spaces in \cref{thm:main}, $\CE(P,M)$ and $\CE(P\times J,M\times J)$, are known to be $(d-p-3)$-connected: the case $P=D^p$ appears in \cite[Proposition 2.6, p.\,26]{BurgheleaLashofRothenberg} and the general case follows from an induction over a handle decomposition.
\item \label{EandCE:iii}The space $\CE(P,M)$ is closely related to the more commonly considered space $\E(P,M)$ of smooth embeddings $e\colon P\hookrightarrow M$ that agree with the inclusion on a neighbourhood of $P\cap\partial M$. Indeed, restriction to $P\times \{1\}$ induces a fibre sequence 
\[
	\E(P\times I, M\times I)\lra \CE(P,M)\lra\E(P,M)
\]
and thus a fibre sequence $\Omega\E(P,M)\ra \E(P\times I, M\times I)\ra \CE(P,M)$. From this point of view, the stabilisation map in \cref{thm:main} can be regarded as a second-order analogue of the map $\Omega \E(P,M)\to \E(P\times I,M\times I)$ that sends a $1$-parameter family of embeddings $P\hookrightarrow M$ indexed by $I$ to the single embedding $P\times I\hookrightarrow M\times I$.
\end{enumerate}
\end{rem}

\begin{figure}
	\begin{tikzpicture}[scale=2.5]
		\draw [Mahogany] (-1.5,0) -- (-1.5,1);
		\node at (-1.5,0) {$\cdot$};
		\node at (-1.5,1) {$\cdot$};
		
		\draw [|->] (-1.25,.5) -- (-.25,.5);
		\node at (-1.5,-.25) {$M$};
		\node at (-1.7,.5) {$I$};

		\begin{scope}[xshift=1cm]
			
 		\draw (-1,0) -- (1,0) -- (1,1) -- (-1,1) -- cycle;
		\foreach \i in {1,...,17}
		{
			\draw [Mahogany] (0,1) ++({-\i*10}:0) -- ++({-\i*10}:1);
		}	
		\draw (1,1) arc [start angle=0, end angle=-180, x 	radius=1cm,y radius =1cm];

		\node at (1.2,.5) {$I$};
		\node at (0,-.25) {$M \times J$};
		\end{scope}

	\end{tikzpicture}
	\caption{The stabilisation map.}
	\label{fig:stabilisation-intro}
\end{figure}
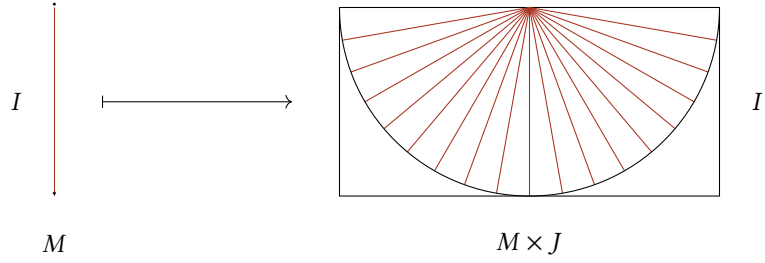

\begin{history}The statement of Theorem \ref{thm:main} is far from new, but a proof has never appeared. It was asserted in the Ph.D.\,thesis \cite[p.\,13]{GoodwillieThesis} of Goodwillie (who apologises for never having produced a proof, cannot recall exactly what kind of proof he had in mind at that time, and acknowledges that the correct number is $2n-p-5$, not $2n-p-4$ as claimed in \cite{GoodwillieThesis}), and then again in the thesis of Meng \cite[Theorem 0.0.1]{Meng}, which contains a proof in the case $P=*$ (see Theorem 3.4.1 loc.cit.). The general case is also referred to at  other places in the literature such as \cite[p.\,6]{IgusaStability} or \cite[p.\,210]{WWsurvey}.
\end{history}

\subsection*{Concordance diffeomorphisms}
One reason to be interested in spaces of concordance embeddings is that they provide information about maps between spaces of \emph{concordance diffeomorphisms} (also called \emph{concordances}, or \emph{pseudoisotopies}), that is, diffeomorphisms of $M\times I$ that agree with the identity on a neighbourhood of $M\times\{0\}\cup \partial M\times I$. The group of such, equipped with the smooth topology, is denoted by $\C(M)$. As for concordance embeddings, there is a stabilisation map $\C(M)\ra \C(M\times J)$. Building on ideas of Hatcher \cite{HatcherSimple}, Igusa proved that this map is approximately $\tfrac{d}{3}$-connected  \cite[p.\,6]{IgusaStability}. This is related to \cref{thm:main} as follows: for a submanifold $P\subset M$ as above, restriction from $M$ to $P$ yields a fibre sequence
\begin{equation}\label{equ:concordance-restriction}
	\C(M\backslash \nu(P))\lra \C(M)\lra \CE(P,M)
\end{equation}
by a variant of the parametrised isotopy extension theorem. Here $\nu(P)\subset M$ is an open tubular neighbourhood of $P$. \cref{thm:main} thus shows that if the handle dimension of the inclusion $\partial M\cap P\subset P$ is at most $d-3$ then the stability range for the base space of \eqref{equ:concordance-restriction} is significantly better than the available stability ranges for the total space and fibre. This can be used to transfer potential results about the homotopy fibre of the stabilisation map for concordance diffeomorphisms of a specific manifold to other manifolds. In \cref{sec:applications} we derive several corollaries from \cref{thm:main} in this direction. Here is an example:

\begin{bigcor}\label{cor:relative-concordance}Let $M$ be a compact $d$-manifold and $D^d\subset \interior(M)$ an embedded disc. If $M$ is $(k-1)$-connected and $k$-parallelisable for some $2 \leq k<\frac{d}{2}$, then the map  
\[
	\pi_i\big(\C(D^d\times J),\C(D^d) \big)\lra\pi_i\big(\C(M\times J),\C(M)\big)
\]
induced by extension by the identity is an isomorphism in degrees $i< d+k-4$.
\end{bigcor}

Here we call a manifold \emph{$k$-parallelisable}, if the restriction of the tangent bundle to a $k$-skeleton is trivialisable. For $k=1,2$ this is the same as being orientable or spin respectively.

\begin{nex}For $k=2$, \cref{cor:relative-concordance} specialises to an isomorphism $\pi_i\big(\C(D^d\times J),\C(D^d) \big)\cong\pi_i\big(\C(M\times J),\C(M)\big)$ for $i<d-2$ and any $1$-connected compact spin $d$-manifold $M$ with $d \geq 5$.
\end{nex}

\noindent Rationally, the relative homotopy groups $\pi_i(\C(D^d\times J),\C(D^d))$ of the stabilisation map for concordance diffeomorphisms of discs have been computed by Krannich and Randal-Williams \cite[Corollary B]{KRWorthogonal} in degrees up to approximately $\tfrac{3}{2}d$. Combined with \cref{cor:relative-concordance}, this gives the following:
 
\begin{bigcor}\label{cor:rational-stable-range}For a compact $(k-1)$-connected $k$-parallisable $d$-manifold $M$ with $2 \leq k <\frac{d}{2}$, there is a homomorphism
\[
	\pi_i\big(\C(M\times J),\C(M)\big)\otimes\bfQ\lra \begin{cases}\bfQ&\text{if }i=d-3,\\
	0&\text{otherwise,}
	\end{cases}
\]
which is an isomorphism in degrees $i<\min(d+k-4,\lfloor\tfrac{3}{2}d\rfloor-8)$ and an epimorphism in degrees $i< \min(d+k-4,\lfloor\tfrac{3}{2}d\rfloor-7)$.
\end{bigcor}

\begin{rem}The assumptions in \cref{cor:rational-stable-range} in particular imply that $M$ is $1$-connected and $d\ge5$, so it follows from the main of result of \cite{Cerf} that $\C(M)$ and $\C(M\times J)$ are connected. This in particular implies that all $\pi_i(\C(M\times J),\C(M))$ are abelian groups, so the rationalisation appearing in the statement of the corollary is unambiguous.
\end{rem}

\noindent Denoting by
\[
	\phi(M)_\bfQ\coloneq \min\left\{s\in\bfZ\ \middle\vert\ \pi_i\big(\C(M\times J^{m+1}), \C(M\times J^m)\big)\otimes\bfQ=0\text{ for }i\le s\text{ and all }m\ge0\right\}
\]
the \emph{rational concordance stable range} of $M$ (the main limiting factor in the classical approach to the rational homotopy type of $\Diff(M)$ through surgery and pseudoisotopy theory, see e.g.\,\cite{WWsurvey}), \cref{cor:rational-stable-range} for $k=2$ in particular implies that $\phi(M^d)_\bfQ=d-4$ for all $1$-connected spin $d$-manifolds $M$ with $d>9$. For $1$-connected spin manifolds, our result confirms speculations of Igusa \cite[p.\,6]{IgusaStability} and Hatcher \cite[p.\,4]{HatcherSurvey} rationally, and improves the ranges of the many results in the literature that rely on the rational concordance stable range. It was known that $\phi(M^d)_\bfQ=d-4$ is the best possible potential result, since work of Watanabe \cite{Watanabe} implies the upper bound $\phi(D^d)_\bfQ\le d-4$ for many odd values of $d$. The previously best known lower bound was $\phi(M)_\bfQ \geq \min(\frac{1}{3}(d-4),\frac{1}{2}(d-7))$, due to Igusa \cite[p.\,6]{IgusaStability}, which is even a lower bound for the integral version of the concordance stable range. 

\subsection*{Acknowledgements}
We thank the anonymous referee for their comments. This material is partially based upon work supported by the Swedish Research Council under grant no.\,2016-06596 while MK and AK were in residence at Institut Mittag-Leffler in Djursholm, Sweden during the semester \emph{Higher algebraic structures in algebra, topology and geometry}. MK was partially funded by the Deutsche Forschungsgemeinschaft (German Research Foundation) under Germany's Excellence Strategy EXC 2044–390685587, Mathematics Münster: Dynamics–Geometry–Structure. AK acknowledges the support of the Natural Sciences and Engineering Research Council of Canada (NSERC) [funding reference number 512156 and 512250], as well as the Research Competitiveness Fund of the University of Toronto at Scarborough. AK is supported by an Alfred J.~Sloan Research Fellowship. 

\section{The multirelative stability theorem and some preliminaries}\label{section:multirelative}
\cref{thm:main} is proved as the case $r=0$ of a multirelative stability theorem that concerns certain $(r+1)$-cubical diagrams of spaces of concordance embeddings; see \cref{thm:main-relative}. The structure of the proof is such that the $r=0$ case requires the general case. In this section we state this multirelative version and establish various preliminaries.

\subsection{Cubical diagrams}\label{sec:cubes}
We begin with a review of cubical diagrams, following \cite{GoodwillieCalculusII}. An \emph{$r$-cube} for $r\ge0$ is a space-valued functor $X$ on the poset category $\cP(S)$ of subsets of a finite set $S$ of cardinality $r$, ordered by inclusion. To emphasise the particular finite set $S$, we sometimes also call $X$ an \emph{$S$-cube}. A $0$-cube is simply a space, a $1$-cube is a map between two spaces, a $2$-cube is a commutative square of spaces, and so on. Since $\varnothing$ is initial in $\cP(S)$ there is a map
\[
	X(\varnothing)\lra \underset{\varnothing\neq T\subseteq S}\holim\,X(T).
\]
The cube $X$ is called \emph{$k$-cartesian} if this map is $k$-connected. Here and throughout this section $k$ may be an integer or $\infty$. For instance, a $0$-cube $X$ is $k$-cartesian if the space $X(\varnothing)$ is $(k-1)$-connected, and a $1$-cube is $k$-cartesian if the map $X(\varnothing)\ra X(S)$ is $k$-connected. As usual, the convention is that a $k$-connected map is in particular surjective on path components if $k\ge 0$, and that a $k$-connected space is non-empty if $k\ge -1$. A \emph{map of $S$-cubes} $X\ra Y$  is a natural transformation. Such a map can also be considered as an $(S\sqcup \{\ast \})$-cube via
\[
	\cP(S\sqcup\{\ast \})\ni T\longmapsto\begin{cases}X(T)&\text{if }\ast \not{\in}T,\\ Y(T\backslash\{\ast \})&\text{otherwise}.
	\end{cases}
\] 
Conversely, an $S$-cube $X$ determines a map of $S\backslash\{\ast \}$-cubes for each $\ast \in S$, and the induced $S$-cube of each of these is isomorphic to $X$. A choice of basepoint $\ast \in X(\varnothing)$ induces compatible basepoints in $X(T)$ for all $T\in\cP(S)$. Given a map of $r$-cubes $X\ra Y$ and a point $\ast \in Y(\varnothing)$, we obtain an $r$-cube $\hofib_\ast (X\ra Y)$ by taking homotopy fibres.

Many standard facts about the connectivity behaviour of maps of spaces generalise to cubes of spaces. For example, from \cite[Propositions 1.6, 1.8, 1.18]{GoodwillieCalculusII} we have:

\begin{lem}\label{lem:cubelemma}For a map $ X\ra Y$ of $r$-cubes, considered as an $(r+1)$-cube, we have:
\begin{enumerate}[label=(\roman*)]
	\item\label{cubelemma:i} If $Y$ and $X\ra Y$ are $k$-cartesian, then $X$ is $k$-cartesian.
	\item\label{cubelemma:ii} If $X$ is $k$-cartesian and $Y$ is $(k+1)$-cartesian, then $ X\ra Y$ is $k$-cartesian.
	\item\label{cubelemma:iii} $X\ra Y$ is $k$-cartesian if and only if $\hofib_\ast (X\ra Y)$ is $k$-cartesian for all $\ast \in Y(\varnothing)$.
\end{enumerate}
Given a further map $Y\ra Z$ of $r$-cubes, considered as an $(r+1)$-cube, we have:
\begin{enumerate}[resume]
	\item\label{cubelemma:iv} If $X\ra Y$ and $Y\ra Z$ are $k$-cartesian, then $X\ra Z$ is $k$-cartesian.
	\item\label{cubelemma:v} If $X\ra Z$ is $k$-cartesian and $Y\ra Z$ is $(k+1)$-cartesian, then $ X\ra Y$ is $k$-cartesian.
\end{enumerate}
\end{lem}

We will also encounter cubes of cubes. Just as a map of $r$-cubes may be considered as an $(r+1)$-cube, an $S'$-cube of $S$-cubes may be considered as an $(S\sqcup S')$-cube, using the canonical isomorphism $\cP(S\sqcup S')\cong \cP(S)\times \cP(S')$ of posets. 

\begin{lem}\label{lem:cube-of-cubes}Let $S$ and $S'$ be non-empty finite sets, and let $X$ be an $(S\sqcup S')$-cube. For $T'\subseteq S'$ write $X_{T'}$ for the $S$-cube given by $\cP(S) \ni T \mapsto X(T\sqcup T')$.
\begin{enumerate}
	\item\label{cube-of-cubes:i} If $X$ is $k$-cartesian and the $S$-cube $X_{T'}$ is $(k+|T'|-1)$-cartesian for all $T'$ such that $\varnothing\neq T' \subseteq S'$, then the $S$-cube $X_\varnothing$ is $k$-cartesian.
	\item\label{cube-of-cubes:ii}  If the $S$-cube $X_{T'}$ is $\infty$-cartesian for all $T'\subseteq S'$, then $X$ is $\infty$-cartesian.
\end{enumerate}
\end{lem}

\begin{proof}
Part \ref{cube-of-cubes:i} is \cite[Proposition 1.20]{GoodwillieCalculusII}. For part \ref{cube-of-cubes:ii}, note that according to \cref{lem:cubelemma} \ref{cubelemma:ii} a map of $\infty$-cartesian $r$-cubes is always an $\infty$-cartesian $(1+r)$-cube. To prove the more general assertion that an $s$-cube of $\infty$-cartesian $r$-cubes is always an $\infty$-cartesian $(s+r)$-cube, we induct on $s$. An $(s+1)$-cube of $\infty$-cartesian $r$-cubes is a map of $s$-cubes of $\infty$-cartesian $r$-cubes, therefore by the inductive hypothesis is a map of $\infty$-cartesian $(s+r)$-cubes, and hence is indeed $\infty$-cartesian.
\end{proof}

\begin{cor}\label{cor:cube-of-cubes}
Let $S$ be a finite set.
\begin{enumerate}
	\item \label{enum:constant-cubes}The constant $S$-cube with value a fixed space $X$ is $\infty$-cartesian as long as $|S|\ge1$.
	\item \label{enum:product-cube} Given spaces $X_s$ for $s\in S$, the $S$-cube
		\[\textstyle{\cP(S)\ni T\mapsto \bigsqcap_{s\in S\backslash T}X_s}\]
	defined by the projections  is $\infty$-cartesian as long as $|S|\ge2$.
	\item \label{enum:pointed-product cube} Given pointed spaces $Y_s$ for $s \in S$, the $S$-cube
		\[\textstyle{\cP(S) \ni T \mapsto \bigsqcap_{s \in T} Y_s}\]
	defined by the inclusions is $\infty$-cartesian as long as $|S| \geq 2$.
\end{enumerate}
\end{cor}

\begin{proof}
When $|S|=1$, part \ref{enum:constant-cubes} is simply the assertion that the identity map $X\ra X$ is $\infty$-connected. The general case of part \ref{enum:constant-cubes} then follows using \cref{lem:cube-of-cubes} \ref{cube-of-cubes:ii}. For part \ref{enum:product-cube}, we pick $\ast \in S$ and view the $S$-cube in question as the map of $(S\backslash \{\ast \})$-cubes $\bigsqcap_{s\in S\backslash \bullet}\,X_s \ra \bigsqcap_{s\in S\backslash (\bullet\cup\{\ast \})}\,X_s$. By \cref{lem:cubelemma} \ref{cubelemma:iii}, it suffices to show that the cubes of homotopy fibres at all basepoints are $\infty$-cartesian. These are constant cubes, so the claim follows from \ref{enum:constant-cubes}. For part \ref{enum:pointed-product cube}, one considers the map of $S$-cubes from the constant cube with value $\bigsqcap_{s \in S} Y_s$ to the cube $\cP(S) \ni T \mapsto \bigsqcap_{s \in S \setminus T} Y_s$ given by the canonical projection maps. By parts \ref{enum:constant-cubes} and \ref{enum:product-cube} both are $\infty$-cartesian, and by \cref{lem:cubelemma} \ref{cubelemma:iii} the same holds for the cube of homotopy fibres over the basepoint provided by the basepoints in the $Y_s$'s. This is the cube in question.
\end{proof}

In a key step of the proof of our main result, we will make use of the following multirelative generalisation of the Blakers--Massey theorem for strongly cocartesian cubes. An $S$-cube $X$ is said to be \emph{strongly cocartesian} if for every subset $T\subseteq S$ and distinct elements $s_1\neq s_2\in S\backslash T$, the following square is homotopy cocartesian
\[
	\begin{tikzcd}
	X(T)\rar \dar&X(T\cup\{s_1\})\dar\\
	X(T\cup \{s_2\})\rar &X(T\cup\{s_1,s_2\}).
	\end{tikzcd}
\]
In terms of this definition, \cite[Theorem\,2.3]{GoodwillieCalculusII} says:

\begin{thm}\label{thm:blakers-massey}
Let $S$ be a non-empty finite set. If $X$ is a strongly cocartesian $S$-cube such that the map $X(\varnothing)\ra X(\{s\})$ is $k_s$-connected for all $s\in S$, then $X$ is $(1-|S|+\sum_{s\in S}k_s)$-cartesian.
\end{thm}

\subsection{The stabilisation map}\label{sec:stabilisation-map}  
We now give a precise definition of the stabilisation map for concordance embeddings. As in the introduction, we fix a smooth manifold $M$ and a compact submanifold $P\subset M$ that meets $\partial M$ transversely. To construct the stabilisation map, we replace $\CE(P,N)$ by the equivalent subspace  
\[
	\CE'(P,N) \subset \CE(P,N)
\]
consisting of those concordance embeddings $e\colon P\times [0,1]\hookrightarrow M\times [0,1]$ such that $e(p,t) = ((\pr_M \circ e)(p,1),t)$ on a neighbourhood of $P \times \{1\}$. Writing $I\coloneq[0,1]$ and $J \coloneq [-1,1]$, we decompose the rectangle $J \times I$ into the two closed subspaces (see \cref{fig:stabilisation})
 \[
 	D_1\coloneqq \{(x,y) \in J \times I \mid x^2+(y-1)^2 \leq 1\} \quad\text{and}\quad
	D_2 \coloneqq \{(x,y)\in J \times I \mid x^2+(y-1)^2 \geq 1\}.
\]
The first of these can be parametrised by polar coordinates via
\begin{align}\label{equ:polar-coordinates}
	\begin{split}\Lambda\colon [0,1] \times [0,\pi]&\longrightarrow D_1\\
	(r,\theta)&\longmapsto \big((1-r)\cos(\theta+\pi),(1-r)\sin(\theta+\pi)+1\big).\end{split}
\end{align}
Writing $e_M \coloneqq (\pr_M \circ e) \colon P \times I \ra M$ and $e_I \coloneqq (\pr_I \circ e) \colon P \times I \ra I$ for a map $e \colon P \times I \to M \times I$ (such as a concordance embedding), the stabilisation map 
\[
	\sigma \colon \CE'(P,M) \lra \CE(P \times J,M \times J)
\]
is defined by sending $e \in \CE'(P,M)$ to the concordance embedding
\begin{equation}\label{equ:stab-map-formula} 
	\begin{aligned}
	\sigma(e)\colon P \times J \times I&\longrightarrow M \times J \times I\\\big(p,s,t\big)&\longmapsto\begin{cases}
	\big(e_M(p,r),\Lambda(e_I(p,r),\theta)\big) & \text{if $(s,t) = \Lambda(r,\theta) \in D_1$,} \\
	\big(p,s,t\big) & \text{if $(s,t) \in D_2$.}\end{cases}
	\end{aligned}
\end{equation}
The point of passing to the subspace $\CE'(P,N)$ is to ensure that $\sigma(e)$ is smooth at $(p,0,1)$.
		
\begin{convention}\label{conv:no-subspace}In what follows we do not distinguish between $\CE(P,M)$ and its homotopy equivalent subspace $\CE'(P,M) \subset \CE(P,M)$. In particular, we write $\CE(P,M)$ for the domain of the stabilisation map, even though it should strictly speaking be $\CE'(P,M)$.

\end{convention}

\begin{figure}
	\begin{tikzpicture}[scale=2.5]
		\draw (-1,0) -- (1,0) -- (1,1) -- (-1,1) -- cycle;
		\foreach \i in {1,...,17}
		{
			\draw [Mahogany] (0,1) ++({-\i*10}:0) -- ++({-\i*10}:1);
		}
		\draw (1,1) arc [start angle=0, end angle=-180, x radius=1cm,y radius =1cm];
		
		\node [fill=white] at (0,.5) {$D_1$};
		\node at (-.75,.16) {$D_2$};
		\node at (.75,.16) {$D_2$};
		
		\node at (1.2,.5) {$I$};
		\node at (0,-.25) {$J$};
	\end{tikzpicture}
	\caption{The decomposition $J \times I = D_1 \cup D_2$. The red arcs indicate the parametrisation \eqref{equ:polar-coordinates}: the semicircle is parametrised by fixing $r=0$ and taking $\theta\in [0,\pi]$ starting with $\theta = 0$ on the left, and the radial segments are parametrised by fixing $\theta \in [0,\pi]$ and taking $r\in [0,1]$ starting with $r=0$ at the semicircle. The map $\sigma(e)$ is given by the identity on $D_2$, and by $e$ on each radial segment in $D_1$.}
	\label{fig:stabilisation}
\end{figure}
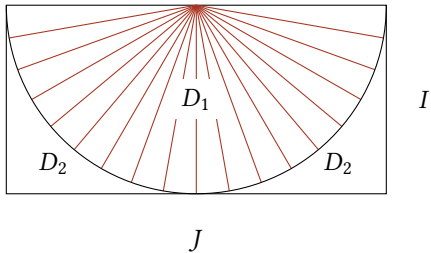

\subsection{Statement of the main theorem}\label{sec:stm-main-thm}
To state the main result, we fix a smooth $d$-manifold $M$ with compact disjoint submanifolds $P, Q_1,\ldots,Q_r\subset M$ for $r\ge0$, all transverse to $\partial M$. We abbreviate $\{1,\ldots,r\}$ by $\ul{r}$ and write $M_S \coloneqq M \backslash  \cup_{i\not\in S}Q_i$ for subsets $S\subseteq \ul{r}$; for example, $M_\varnothing = M\backslash{\cup_{i =1}^r Q_i}$ and $M_{\ul{r}} = M$. Postcomposition with the inclusions $M_S\subset M_{S'}$ for $S\subset S'$ induces inclusions $\CE(P,M_S)\subset \CE(P,M_{S'})$ that assemble to an $r$-cube $\CE(P,M_\bullet)$. As the construction of the stabilisation map from \cref{sec:stabilisation-map} is natural in inclusions of submanifolds $M\subset M'$ with $P\cap\partial M=P\cap \partial M'$, it extends to a map of $r$-cubes $\CE(P,M_\bullet)\ra \CE(P\times J,M_\bullet \times J)$ which we view as an 
$(r+1)$-cube 
\[
	\sCE(P,M_\bullet)\coloneq\big(\CE(P,M_\bullet)\ra \CE(P\times J,M_\bullet \times J)\big).
\] 
The main result of this work concerns this $(r+1)$-cube. It includes \cref{thm:main} as the case $r=0$. The statement involves the dimension $d$ of $M$ and the handle dimensions $p, q_i$ of the inclusions $\partial M\cap P\subset P$ and $\partial M \cap Q_i\subset Q_i$. The numbers $q_i$ will play a role via the quantity
\[
	\textstyle{\Sigma\coloneq \sum_{i=1}^r(d-q_i-2)},
\]
which we abbreviate as indicated since it will be ubiquitous in everything that follows.

\begin{thm}\label{thm:main-relative}
If $d-p\ge3$ and $d-q_i\ge3$ for all $i$, then the $(r+1)$-cube 
\[
	\sCE(P,M_\bullet)=\big(\CE(P,M_\bullet)\ra \CE(P\times J,M_\bullet\times J)\big)
\] 
is $(2d-p-5+\Sigma)$-cartesian. \end{thm}

The proof occupies \cref{sec:proof}. The remainder of this section contains more preliminaries.

\subsection{Collars and tubular neighborhoods}\label{sec:disc-bundles-dont-matter} 
The following two lemmas describe the homotopy type of $\CE(P,M)$ for some simple choices of $P$.

\begin{lem}\label{lem:collar}
If $P$ is a closed collar on $P\cap \partial M\subset P$, then $\CE(P,M)$ is contractible.
\end{lem}

\begin{proof}
In view of \cref{EandCE} \ref{EandCE:iii} it suffices to show that $\E(P,M)$ and $\E(P\times I,M\times I)$ are contractible. This follows from the contractibility of the space of collars.
\end{proof}

\begin{lem}\label{lem:disc-bundle}
For a closed disc-bundle $\pi\colon D(P)\ra P$ with an embedding $D(P) \hookrightarrow M$ that extends the inclusion $P \subset M$ and satisfies $D(P)\cap \partial M=\pi^{-1}(P\cap \partial M)$, the map $\CE(D(P),M) \to \CE(P,M)$ induced by restriction to the $0$-section is an equivalence.
\end{lem}

\begin{proof}
The homotopy fibre of the restriction map $\CE(D(P),M)\to \CE(P,M)$ over $e \in  \CE(P,M)$ agrees as a result of the parametrised isotopy extension theorem with the strict fibre. The latter is, by taking derivatives, equivalent to the space of sections over $P\times I$ fixed near $P \times \{0\}\cup (\partial M\cap P)\times I$ of the bundle of fibrewise linear injections of $\nu_{D(P) \times I}$ into $\nu_{M \times I}$ over $P \times I$; here $\nu_{D(P) \times I}$ is the normal bundle of $P \times I\subset D(P) \times I$ and $\nu_{M \times I}$ is the normal bundle of $e(P \times I)\subset M \times I$. Since $P \times \{0\}\cup (\partial M\cap P)\times I$ is a deformation retract of $P\times I$, this section space is contractible. 
\end{proof}

\subsection{Previous multiple disjunction results} \label{sec:multiple-disjunction}
A key ingredient in the proof of the multirelative stabilisation result in \cref{thm:main-relative} is the following multirelative generalisation of Morlet's lemma of disjunction from \cite{Goodwillie}. We use the notation from \cref{sec:stm-main-thm}.

\begin{thm}[Goodwillie]\label{thm:multipl-disjunction}
If $d-p\ge3$ and $d-q_i\ge3$ for all $i$, then the $r$-cube $\CE(P,M_\bullet)$ is $(d-p-2+\Sigma)$-cartesian.
\end{thm}

\begin{proof}
We already discussed this for $r=0$ in \cref{EandCE} \ref{EandCE:ii}. The case $r\ge1$ is treated in \cite[Theorem D]{Goodwillie}. There it is assumed that $\partial M\cap P=\partial P$ and $\partial M\cap Q_i=\partial Q_i$ for all $i$, but as pointed out in \cite[p.\,670]{GoodwillieKlein} the general case can be easily be reduced to this.
\end{proof}

For $r=0$, this statement includes Hudson's concordance-implies-isotopy theorem for concordance embeddings \cite[Theorem 2.1, Addendum 2.1.2]{HudsonConcordance}.

\begin{thm}[Hudson]\label{thm:hudson} 
The space $\CE(P,M)$ is connected if $d-p\ge3$.
\end{thm}

Using \cref{thm:multipl-disjunction}, Goodwillie and Klein proved a similar result for spaces of ordinary embeddings \cite[Theorem A]{GoodwillieKlein}. As in \cref{EandCE} \ref{EandCE:iii}, we write $\E(P,M)$ for the space of embeddings $P\hookrightarrow M$ that coincide with the inclusion in a neighbourhood of $P\cap\partial M$. As with $\CE(P,M_\bullet)$, the spaces $\E(P,M_S)$ for subsets $S\subseteq\ul{r}$ assemble to an $r$-cube $\E(P,M_\bullet)$.

\begin{thm}[Goodwillie--Klein]\label{thm:multipl-disjunction-emb}
If $d-p\ge3$ and $d-q_i\ge3$ for all $i$, and if $r\ge 1$, then the $r$-cube $\E(P,M_\bullet)$ is $(1-p+\Sigma)$-cartesian.
\end{thm}

The cube $\E(P,M_\bullet)$ appearing in \cref{thm:multipl-disjunction-emb} arises by removing submanifolds from the target, but there is also a version of this theorem that deals with removing submanifolds from the source \cite[Theorem C]{GoodwillieKlein}. We will have use for a version that combines these two. To state it, in addition to $P,Q_1,\ldots,Q_r \subset M$, we also fix disjoint compact codimension $0$ submanifolds $B_1,\ldots,B_k \subset P$ for $k \geq 0$, all transverse to $\partial P$. We write $P_T$ for the closure of $P\backslash  \cup_{j \in T} B_j$ for $T\subseteq \ul{k}$; for example, $P_\varnothing = P$ and $P_{\ul{k}}$ is the closure of $P \backslash{\cup_{j=1}^k B_j}$. The spaces $\E(P_T,M_S)$ assemble into a $(k+r)$-cube $\E(P_{\smallsquare},M_\bullet)$ by post- and precomposition with the inclusions. In addition to the handle dimension $q_i$ of the inclusion $\partial M\cap Q_i\subset Q_i$, we write $b_j$ for the handle dimension of the inclusion $\partial B_j\backslash( \partial P \cap B_j)\subset B_i$ and abbreviate $\Sigma' = \sum_{j=1}^k (d-b_j-2)$. The 
following can be deduced from \cref{thm:multipl-disjunction-emb} by a simple variant of the arguments from \cite[p.~653--655]{GoodwillieKlein}. 

\begin{cor}\label{cor:disjunction-excision} If $d-p \geq 3$, $d-q_i \geq 3$ for all $i$, and $d-b_j \geq 3$ for all $j$, and if $k+r\ge 2$, then the $(k+r)$-cube $\E(P_{\smallsquare},M_\bullet)$ is $(3-d+\Sigma+\Sigma')$-cartesian.
\end{cor}

\begin{rem}\label{rem:analyticity}In the language of the functor calculus of \cite{WeissImmersion,GoodwillieWeiss,WeissImmersionErrata} (``manifold calculus'' or ``embedding calculus''), the multirelative connectivity results Theorems \ref{thm:multipl-disjunction} and \ref{thm:multipl-disjunction-emb} can be viewed as analyticity statements for the functors
\[
P\mapsto \CE(P,M)\quad\text{and}\quad P\mapsto \E(P,M).
\]
defined on the poset of compact submanifolds of a fixed manifold $M$ (or rather, in a setting close to \cite[Section\,2]{GoodwillieWeiss}, for the analogous functors defined on the poset of open subsets of $M$). For $\CE$, this results by stabilisation in an analyticity statement in the sense of \cite{GoodwillieCalculusII} and \cite{GoodwillieCalculusIII} (``homotopy calculus'') for the functor given by stable concordance theory (there is also a different way to prove that using \cite[Theorem 4.6]{GoodwillieCalculusII}). Our main result---the multirelative stability result \cref{thm:main-relative}---can be viewed an analyticity statement for the functor
\[
P\mapsto \hofib_\inc\big(\CE(P,M)\to \CE(P\times J,M\times J)\big).
\]
All these analyticity results have the same degree of analyticity, but differ in ``excess''.
\end{rem}

\subsection{The delooping trick and scanning}\label{sec:delooping}
We now explain a way to relate concordance embeddings of discs of different dimensions, sometimes called the delooping trick. It goes back at least to \cite[23-25]{BurgheleaLashofRothenberg}. For an embedded disc $D^p\subset M$ with $D^p\cap\partial M=\partial D^p$ and $p\ge 1$, we first define a \emph{scanning map} of the form
\begin{equation}\label{equ:scanning-map}
	\tau \colon \CE(D^p,M) \lra \Omega \CE(D^{p-1},M).
\end{equation}
For a submanifold $K\subset [-1,1]$ we abbreviate
\[D^p_K\coloneq \{x\in D^p\mid x_1\in K\}\subset D^p,\]
so in particular $D^p_{\{0\}}=\{0\}\times D^{p-1}\cong D^{p-1}$. To construct \eqref{equ:scanning-map}, we consider the decomposition $\smash{D^p=D^p_{[-1,0]}\cup_{D^p_{\{0\}}} D^p_{[0,1]}}$ and the resulting commutative diagram of restriction maps
\begin{equation}\label{equ:delooping-square}
	\begin{tikzcd}\CE(D^p,M) \rar \dar & \CE(D^p_{[0,1]},M) \dar \\
	\CE(D^p_{[-1,0]},M) \rar & \CE(D^p_{\{0\}},M)\end{tikzcd}
\end{equation}
which induces a map from $\CE(D^p,M)$ to the homotopy pullback of the other terms. The bottom-left and top-right terms are contractible by \cref{lem:collar}, so the homotopy pullback is equivalent to the loop space of $\smash{\CE(D^p_{\{0\}},M)\cong\CE(D^{p-1},M)}$. This defines the scanning map \eqref{equ:scanning-map}, up to contractible choices. More precisely, writing $\smash{\widetilde{\Omega}\CE(D^{p-1},M)}$ for the homotopy pullback, we have (without choices) a zig-zag
\begin{equation}\label{equ:scanning-zig-zag}
	\CE(D^p,M)\lra \widetilde{\Omega}\CE(D^{p-1},M)\xlla{\simeq}\Omega\CE(D^{p-1},M)
\end{equation}
where the equivalence is induced by including the basepoints in the spaces $\smash{\CE(D^p_{[0,1]},M)}$ and $\smash{\CE(D^p_{[-1,0]},M)}$. The map \eqref{equ:scanning-map} can also be viewed as the map on vertical homotopy fibres in \eqref{equ:delooping-square}, where the fibres are taken over the basepoints given by the inclusions of $\smash{D^p_{[-1,0]} \times I }$ and $\smash{D^p_{\{0\}} \times I}$ into $M \times I$. This map on vertical homotopy fibres is equivalent (as a result of the parametrised isotopy extension theorem) to the inclusion 
\[
	\CE\big(D^p_{[\epsilon,1]},M\backslash ( T\cup D_{[-1,-\epsilon]}^p))\subset \CE\big(D^p_{[\epsilon,1]},M\backslash T),
\]
where $\epsilon\in(0,1)$ and $T$ is an open tubular neighbourhood of $\smash{D^p_{\{0\}}}\subset M$ with (see \cref{fig:handle-splitting})  
\[
	\smash{D^p_{[\epsilon,1]}\cap\partial (M\backslash T)=\partial D^p_{[\epsilon,1]}}\quad\text{and}\quad\smash{D^p_{[-1,-\epsilon]}\cap\partial (M\backslash T)=\partial D^p_{[-1,-\epsilon]}}.
\]

Turning to the multirelative setting of \cref{sec:stm-main-thm}, we note that the zig-zag \eqref{equ:scanning-zig-zag} is natural in inclusions of submanifolds $M\subset M'$ with $\partial M\cap P=\partial M'\cap P$, so that up to contractible choices we have a scanning map of $r$-cubes
\begin{equation}\label{scanning-cube}
	\tau\colon\CE(D^p,M_\bullet) \lra \Omega \CE(D^{p-1},M_\bullet)
\end{equation}
that agrees up to equivalence with the inclusion of $r$-cubes \begin{equation}\label{eqref:inclusion-delooping}
	\CE\big(D^p_{[\epsilon,1]},M_\bullet\backslash ( T\cup D_{[-1,-\epsilon]}^p))\subset \CE\big(D^p_{[\epsilon,1]},M_\bullet\backslash T).
\end{equation}

\begin{figure}\begin{tikzpicture}[scale=1.2]
		\draw (0,0) circle (2cm);
		\node at (0,2) [above] {$M$};
		\draw [thick,Periwinkle,fill=Periwinkle!15!white] (2,0) arc (0:360:2 and 0.6);
		\draw [fill=white] (.3,.6) circle (.3cm);
		\draw [fill=white] ({.3+.3*cos(145)},{.6+.3*sin(145)}) -- ({-.3+.3*cos(145)},{-.55+.3*sin(145)}) -- ({-.3+.3*cos(180+140)},{-.55+.3*sin(180+140)}) -- ({.3+.3*cos(180+145)},{.6+.3*sin(180+145)});
		\draw [fill=white](-.3,-.55) circle (.3cm);
		\draw [thick,Mahogany] (-.31,-.58) -- (.3,.6);
		
		\begin{scope}
			\clip (.3,.6) circle (.3cm);
			\draw [dashed] (2,0) arc (0:180:2 and 0.6);
		\end{scope}
		
		\begin{scope}
			\clip (-.3,-.55) circle (.3cm);
			\draw [dashed] (2,0) arc (0:-180:2 and 0.6);
		\end{scope}
		
		\node [Periwinkle] at (-2.6,1.2) {$D^p_{[-1,-\epsilon]}$};
		\draw [->,Periwinkle] (-2.12,1.25) to[out=0,in=100] (-1,0);
		
		\node [Periwinkle] at (2.6,1.2) {$D^p_{[\epsilon,1]}$};
		\draw [->,Periwinkle] (2.12,1.25) to[out=180,in=80] (1,0);
		
		\node [Mahogany] at (-2.5,-1.2) {$D^{p-1}$};
		\draw [->,Mahogany] (-2.12,-1.25) to[out=0,in=-135] (-.37,-.62);
		\node  at (.35,1.1) {$T$};
	\end{tikzpicture}
	\caption{The subspaces of $M$ appearing in the delooping trick.}
	\label{fig:handle-splitting}
\end{figure}
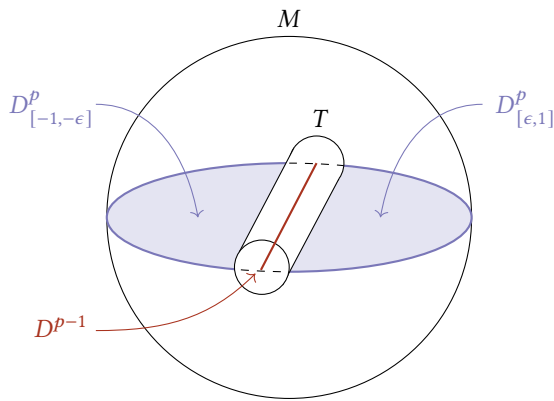

\begin{lem}\label{lem:delooping} 
If $d-p \geq 3$ and $d-q_i \geq 3$ for all $i$, then map of $r$-cubes
\[
	\tau\colon\CE(D^p,M_\bullet) \lra \Omega \CE(D^{p-1},M_\bullet)
\]
is $(2\cdot(d-p-2)+\Sigma)$-cartesian when considered as an $(r+1)$-cube.
\end{lem}

\begin{proof}
This $(r+1)$-cube may be rewritten as $\smash{\CE(D^p_{[\epsilon,1]},(M\backslash T)_\bullet)}$ with $Q_1,\ldots,Q_r \subset M\backslash T$ as before and $Q_{r+1}\coloneq D_{[-1,-\epsilon]}^p$, so it is 
$(d-p-2+\Sigma+(d-p-2))$-cartesian by \cref{thm:multipl-disjunction}.
\end{proof}

\subsection{Concordance maps and immersions}\label{sec:immersions}It will be useful to compare spaces of concordance embeddings to spaces of concordance maps and of concordance immersions. The space $\CF(P,M)$ of \emph{concordance maps} is the space of smooth maps $\varphi \colon P \times I \rightarrow M \times I$ that satisfies the following two conditions:
\begin{enumerate}
	\item\label{first} $\varphi^{-1}(M \times \{i\}) = P \times \{i\}$ for $i = 0,1$ and 
	\item \label{second}$\varphi$ agrees with the inclusion on a neighbourhood of $P \times \{0\} \cup (\partial M \cap P) \times I \subset M \times I$,
\end{enumerate} 
equipped with the smooth topology. The space of \emph{concordance immersions}  $\CI(P,M)\subset \CF(P,M)$ is the subspace of those maps that are immersions. Note that we have inclusions
\[
	\CE(P,M)\subset \CI(P,M)\subset \CF(P,M).
\]
As we shall explain now, the homotopy types of both $\CI(P,M)$ and $\CF(P,M)$ are significantly simpler than that of $\CE(P,M)$. We begin with $\CF(P,M)$:

\begin{lem}\label{lem:CF-contractible}$\CF(P,M)$ is contractible.\end{lem}

\begin{proof}Given a concordance map $f\colon P\times I\rightarrow M\times I$, the family of concordance maps 
\[
	f_s(p,t)=\big((\pr_M\circ f)(p,t),(1-s)\cdot (\pr_I\circ f)(p,t)+s\cdot t\big) \qquad \text{for $s\in[0,1]$}
\]
defines a deformation retraction of $\CF(P,M)$ onto the subspace $\CF(P,M)_I\subset \CF(P,M)$ of concordance maps that are \emph{level-preserving}, i.e.\,commute with the projection onto $I$. The space $\CF(P,M)_I$ further deformation-retracts onto the basepoint by the family of paths  
\[
	f_s(p,t)=\big((\pr_M\circ f)(p,(1-s)\cdot t),t\big) \qquad \text{for $s \in [0,1]$}.\qedhere
\]\end{proof}

Turning to the space $\CI(P,M)$ of concordance immersions, we assume that the handle dimension $\partial M \cap P\subset P$ is less than $d$. Under this assumption, by Smale--Hirsch theory, differentiation gives an  equivalence
\[
	\CI(P,M)\xlra{\simeq} \CB(P,M),
\]
where $\CB(P,M)$ is the space of \emph{concordance bundle maps}, equipped with the compact-open topology. A concordance bundle map is a fibrewise injective vector bundle map $T(P\times I)\rightarrow T(M\times I)$ covering a concordance map $f$ such that for points $(x,t)$ near $P \times \{0\} \cup (\partial M \cap P) \times I \subset P \times I$ the linear map $T_{(x,t)}(P\times I)\to T_{(x,t)}(M\times I)$ is the inclusion, and such that at points $(x,1)$ the linear map $T_{(x,1)}(P\times I)\to  T_{f(x,1)}(M\times I)$ takes the subspace $TP\times 0$ into the subspace $TM\times 0$ and is positive in the $I$-direction. This can also be described as a space of sections:

\begin{lem}\label{lem:cimm-as-sections}If the handle dimension of $\partial M \cap P\subset P$ is less than $d$, then there are equivalences
\[
	\Sect_{P\cap \partial M}(\Omega S^{d}\times_{\oO(d)} \Fr(M)|_P \to P) \overset{\simeq}\lra \CB(P,M) \overset{\simeq}\longleftarrow \CI(P,M).
\]
Here $\Fr(M)|_P$ is the restriction to $P$ of the frame bundle of $TM$, $\oO(d)$ acts on $S^d$ via the one-point compactification of $\bfR^d$, and $\Sect_{P\cap \partial M}(-)$ stands for the space of sections of the indicated bundle that agree with the standard section in a neighbourhood of $P\cap \partial M$.
\end{lem}

\begin{proof}
We have already explained the right-hand equivalence. For the other, note that the forgetful map $\CB(P,M)\ra \CF(P,M)$ is a fibration whose base space is contractible by \cref{lem:CF-contractible}, so it suffices to show that the indicated section space is equivalent to the fibre $\CB_{\inc}(P,M)$ over the basepoint $\inc\in\CF(P,M)$. 

A bundle map that covers the inclusion is given by linear injections $T_{(x,t)}(P\times I)\ra T_{(x,t)}(M\times I)$ for $(x,t)\in P\times I$ satisfying certain boundary conditions. Fixing $x$, varying $t$, and using the standard trivialisation of $TI$, this becomes a path in the space $\Inj(T_xP\oplus \bfR,T_xM\oplus \bfR)$ of linear injections, starting at the inclusion and ending somewhere in $\Inj(T_xP,T_xM)$ (viewed as a subspace of $\Inj(T_xP\oplus \bfR,T_xM\oplus \bfR)$ via $(-)\oplus \id_\bfR$). From this we see that $\CB_{\inc}(P,M)$ is the space of sections, trivial near $P\cap \partial M$, of a bundle on $P$ whose fibre over $x\in P$ is $F_x\coloneq\hofib_{\inc}(\Inj(T_xP,T_xM)\ra\Inj(T_xP\oplus \bfR,T_xM\oplus \bfR))$. The fibre sequence
\vspace{-0.1cm}
\[
	\Inj(T_xP,T_xM)\lra \Inj(T_xP\oplus\bfR ,T_xM\oplus\bfR)\xlra{\res} \Inj(\bfR,T_xM\oplus \bfR)\cong S^{T_xM}
\]
gives an equivalence $F_x\simeq \Omega S^{T_xM}$ to the loop space on the one-point compactification of $T_xM$. This depends continuously on $x$, so $\CB_{\inc}(P,M)$ is equivalent to the space of sections, trivial near $P\cap\partial M$, of a bundle whose fibre is $\Omega S^{T_x M}$. This bundle is $\Omega S^d\times_{\oO(d)} \Fr(M)|_P \to P$.
\end{proof}

Note that the equivalences in the proof of \cref{lem:cimm-as-sections} are natural in codimension $0$ embeddings $e\colon M\hookrightarrow M'$ with $P\cap \partial M'=P\cap \partial e(M)$. In particular, we can conclude:

\begin{lem}\label{lem:cimm-independent-of-target}Assume that the handle dimension of $\partial M \cap P\subset P$ is less than $d$. For an open neighbourhood $U\subset M$ of $P$, the inclusion $\CI(P,U) \subset\CI(P,M)$ is an equivalence. 
\end{lem}

\begin{rem}\label{rem:immersion-point}\ 
\begin{enumerate}
\item Note that the fibre sequence at the end of the proof of \cref{lem:cimm-as-sections} in particular shows that the map $F_x\ra \CB_\inc(\{x\},M)$ is an equivalence.
\item \label{enum:immersion-point-ii}For $P = \ast\in\interior(M)$, \cref{lem:cimm-as-sections} gives a $\Diff_\ast(M)$-equivariant equivalence $\CI(\ast,M)\simeq \Omega S^{T_\ast M}$ to the loop space on the one-point compactification of $T_\ast M$. Applying this for $M=\bfR^d$, we see that the equivalence of \cref{lem:cimm-as-sections} can also be written in the form $\CI(P,M)\simeq \Sect_{P\cap \partial M}(\CI(\ast,\bfR^d)\times_{\oO(d)} \Fr(M)|_P \to P)$.
\end{enumerate}
\end{rem}

\subsubsection{The stabilisation and scanning maps for concordance maps and immersions}
The construction of the stabilisation and scanning map in Sections~\ref{sec:stabilisation-map} and \ref{sec:delooping} extend to concordance maps, concordance immersions, and concordance bundle maps, so there are commutative diagrams 
\[\begin{tikzcd}[ar symbol/.style = {draw=none,"\textstyle#1" description,sloped},
  subset/.style = {ar symbol={\subset}}, row sep=0.2cm]
\CE(P,M)\arrow[d,subset]\rar{\sigma}&\CE(P\times J,M\times J)\arrow[d,subset]&&\CE(D^p,M)\rar\arrow[d,subset]\arrow[r,"\tau"]&\Omega\CE(D^{p-1},M)\arrow[d,subset]\\
\CI(P,M)\dar\rar{\sigma}&\CI(P\times J,M\times J)\arrow[d]&&\CI(D^p,M)\rar\dar\arrow[r,"\tau"]&\Omega\CI(D^{p-1},M)\arrow[d]\\
\CB(P,M)\dar\rar{\sigma}&\CB(P\times J,M\times J)\dar&&\CB(D^p,M)\rar\dar\arrow[r,"\tau"]&\Omega\CB(D^{p-1},M)\dar\\
\CF(P,M)\rar{\sigma}&\CF(P\times J,M\times J)&&\CF(D^p,M)\rar\arrow[r,"\tau"]&\Omega\CF(D^{p-1},M).
\end{tikzcd}.\]
In all cases except for the stabilisation map for concordance bundle maps, the construction is exactly the same as for concordance embeddings. In the remaining case, it is helpful to note that the concordance $\sigma(e)\in\CE(P\times J, M\times J)$ for $e\in\CE(P,M)$ can be described as the unique continuous map $P\times J\times I\ra M\times J\times I$ that agrees with the inclusion on $P\times D_2$ and with the composition $(\id_M\times \Lambda^{-1})\circ (e\times \id_{[0,\pi]})\circ (\id_P\times \Lambda)$ on $M\times (D_1\backslash{\{(0,1)\}})$, using that $\Lambda$ restricts to a diffeomorphism $D_1\backslash{\{(0,1)\}}\cong [0,1)\times [0,\pi]$. Said like this, the definition makes equal sense for concordance bundle maps.

The sources and targets of scanning and stabilisation for concordance maps are contractible by \cref{lem:CF-contractible}, so these maps are equivalences. The scanning map for concordance immersions is also an equivalence, though for a different reason:

\begin{lem}\label{lem:cimm-scanning-equivalence}For $p<d$, the map $\tau\colon \CI(D^p,M)\ra\Omega\CI(D^{p-1},M)$ is an equivalence.
\end{lem} 

\begin{proof}
By the construction of the scanning map in \cref{sec:delooping}, it suffices to show that the square induced by restriction maps
\[
\begin{tikzcd}
\CI(D^p,M)\rar\dar&\CI(D^p_{[0,1]},M)\dar\\
\CI(D^p_{[-1,0]},M)\rar&\CI(D^p_{\{0\}},M)
\end{tikzcd}
\]
 is homotopy cartesian. Via the natural equivalences of \cref{lem:cimm-as-sections} and the standard trivialisation of $TD^d$, this translates to the claim that the square of mapping spaces
\[
\begin{tikzcd}
\Map_\partial(D^p,\Omega S^d)\rar\dar&\Map_{D^p_{[0,1]}\cap\partial D^p} (D^p_{[0,1]},\Omega S^d)\dar\\
\Map_{D^p_{[-1,0]}\cap\partial D^p} (D^p_{[-1,0]},\Omega S^d)\rar&\Map_\partial(D^p_{\{0\}},\Omega S^{d})
\end{tikzcd}
\]
induced by restriction is homotopy cartesian. This square agrees with the induced square on homotopy fibres of the map of squares from
\[
\begin{tikzcd}[column sep=0.1cm]
\Map(D^p,\Omega S^d)\rar\dar &\Map(D^p_{[0,1]},\Omega S^d)\dar\\
\Map(D^p_{[-1,0]},\Omega S^d)\rar &\Map(D^{p}_{\{0\}},\Omega S^d)
\end{tikzcd}\text{\ \ to\ \ }
\begin{tikzcd}[column sep=-0.3cm]
\Map(\partial D^p,\Omega S^d)\rar \dar&\Map(D^p_{[0,1]}\cap \partial D^p,\Omega S^d)\dar\\
\Map(D^p_{[-1,0]}\cap \partial D^p,\Omega S^d)\rar & \Map(\partial D^{p}_{\{0\}},\Omega S^d)
\end{tikzcd}
\]
induced by restriction. This implies the claim, since both of these squares are homotopy cartesian, being obtained by applying $\Map(-,\Omega S^d)$ to a homotopy cocartesian square.
\end{proof}

The stabilisation map for concordance immersions is not an equivalence in general, but we have the following connectivity estimate:

\begin{lem}\label{lem:cimm-freudenthal}If the handle dimension of $\partial M \cap P\subset P$ is less than $d$, then the stabilisation map $\CI(P,M) \to \CI(P \times J,M \times J)$ is $(2d-p-2)$-connected.\end{lem}

\begin{proof}Arguing as in the proof of \cref{lem:cimm-as-sections}, one sees that the stabilisation map agrees up to equivalence with a map between two section spaces relative to $\partial M\cap P$ of bundles over $P$ whose induced map between the fibres over $\ast\in P$ is the stabilisation map $\CB_{\inc}(\ast,M)\ra \CB_{\inc}(\ast\times J,M\times J)$ between the spaces of concordance bundle maps covering the inclusion. By obstruction theory, it thus suffices to show that this map on fibres is $(2d-2)$-connected. 

To do so, we first replace $\CB_{\inc}(\ast\times J,M\times J)$ by an equivalent space in two steps. Firstly, by the proof of \cref{lem:cimm-as-sections}, the space $\CB_{\inc}(\ast\times J,M\times J)$ is a section space of a bundle over $J$ associated to the frame bundle of $J$, so making use of the standard trivialisation of $TJ$, the space $\CB_{\inc}(\ast\times J,M\times J)$ is homeomorphic to $\Map_\partial(J,F_0)$ where $F_0$ is the fibre of the bundle over $0 \in J$. The latter admits a canonical equivalence to $\CB_{\inc}(\ast\times \{0\},M\times J)$ (see the first part of \cref{rem:immersion-point}), so we have an equivalence $\CB_{\inc}(\ast\times J,M\times J)\ra \Map_{\partial}(J,\CB_{\inc}(\ast\times \{0\},M\times J))$. Secondly, we replace $\Map_\partial(J,\CB_{\inc}(\ast\times\{0\},M\times J))$ by the equivalent space $\Map^{\pm}(J,\CB_{\inc}(\ast\times\{0\},M\times J))$ of those paths $[-1,1]=J\ra \CB_{\inc}(\ast,M\times J)$ that start somewhere in the subspace $\CB^+_{\inc}(\ast\times\{0\},M\times J)\subset \CB_{\inc}(\ast\times\{0\},M\times J)$ of concordance bundle maps $TI\ra T(M\times J\times I)$ covering the inclusion that land in the subspace of $T(M\times J\times I)$ whose tangent vector of the $J$-factor is nonnegative, and that end somewhere in the subspace $\CB^-_{\inc}(\ast\times\{0\},M\times J)\subset \CB_{\inc}(\ast,M\times J)$ which is defined similarly by replacing ``nonnegative'' with ``nonpositive''. Note that the intersection of these two subspaces agrees with $\CB_{\inc}(\ast,M
)\subset \CB_{\inc}(\ast\times\{0\},M\times J)$ which contains the basepoint $\inc\in \CB_{\inc}(\ast\times\{0\},M\times J)$. We thus have an inclusion map $\Map_\partial(J,\CB_{\inc}(\ast\times\{0\},M\times J))\subset \Map^{\pm}(J,\CB_{\inc}(\ast\times\{0\},M\times J))$, which is an equivalence since $\CB^+_{\inc}(\ast\times\{0\},M\times J)$ and $\CB^-_{\inc}(\ast\times\{0\},M\times J)$ are both contractible. It thus suffices to show that the composition
\vspace{-0.1cm}\[\hspace{-0.15cm}\CB_{\inc}(\ast,M)\xra{\sigma} \CB_{\inc}(\ast\times J,M\times J)\xra{\simeq} \Map_\partial(J,\CB_{\inc}(\ast\times\{0\},M\times J))\subset \Map^{\pm}(J,\CB_{\inc}(\ast\times\{0\},M\times J))\] is $(2d-2)$-connected. We denote this composition by $\sigma'$. Tracing through the definitions, one sees that the path $\sigma'(f)\colon J=[-1,1]\ra \CB_{\inc}(\ast\times\{0\},M\times J)$ for $f\in \CB_{\inc}(\ast,M)$ satisfies $\sigma'(f)_s\in \CB^+_{\inc}(\ast\times\{0\},M\times J)$ for $s\in[-1,0]$ and $\sigma'(f)_s\in \CB^-_{\inc}(\ast\times\{0\},M\times J)$ for $s\in[0,1]$, so we can define a homotopy \[[0,1]\times \CB_{\inc}(\ast,M)\ra \Map^{\pm}(J,\CB_{\inc}(\ast\times\{0\},M\times J))\] by sending $(u,f)$ to the path $[-1,1]\ni s\mapsto \rho(f)_{(1-u)s}\in\CB_{\inc}(\ast\times\{0\},M\times J)$. This homotopy starts at $\sigma'$ and ends at the map that sends $f$ to the constant path at $\sigma'(f)_0$. The latter agrees with the canonical map from the top-left corner of the commutative square
\[\begin{tikzcd}
\CB_{\inc}(\ast,M)\rar\dar&\CB_{\inc}^+(\ast\times\{0\},M\times J)\dar\\
\CB_{\inc}^-(\ast\times\{0\},M\times J)\rar&\CB_{\inc}(\ast\times\{0\},M\times J)
\end{tikzcd}
\]
to the homotopy pullback of the remaining entries, so we need to show that this square is $(2d-2)$-cartesian. Using the canonical trivialisation of $TI$, one sees that the space $\CB_{\inc}(\ast\times\{0\},M\times J)$ is the loop space $\Omega \Inj(\bfR,T_{(\ast,0)}(M\times J)\oplus \bfR)$ which is equivalent to $\Omega S(T_{(\ast,0)}(M\times J)\oplus \bfR)\simeq \Omega S^{d+1}$. This equivalence (or rather the proof of it) gives an equivalence of squares from the previous square to the square obtained by looping once the cocartesian square
\[\begin{tikzcd}
S^d\rar\dar&S^{d+1}_+\dar\\
S^{d+1}_-\rar&S^{d+1}
\end{tikzcd}
\]
where $S^{d+1}_{\pm}\subset S^{d+1}$ are the left and right hemispheres. By Freudenthal's suspension theorem, this  square is $(2d-1)$-cartesian, so looping it indeed results in a $(2d-2)$-cartesian square.
\end{proof} 

\section{The proof of the multirelative stability theorem}\label{sec:proof}
It is time to turn to the proof of the main result, \cref{thm:main-relative}. Most of the work goes into the case when $P$ is a point (see \cref{section:stability-point}). The case when $P$ is a $p$-disc $D^p$ with $\partial D^p=D^p\cap\partial M$ then follows by induction on $p$ using multirelative disjunction and the delooping trick (see \cref{sec:proof-disc}). The general case follows by induction over a handle decomposition (\cref{sec:proof-general}).

\subsection{The case of a point}\label{section:stability-point}
When $P$ is a point $\ast \in\interior(M)$, the asserted conclusion of \cref{thm:main-relative} is that the stability $(r+1)$-cube 
\begin{equation}\label{equ:stabilisation-point}
	\sCE(\ast,M_\bullet)=\big(\CE(\ast,M_\bullet)\xrightarrow{\sigma} \CE(\ast\times J,M _\bullet \times J)\big)
\end{equation}
is $(2d-5+\Sigma)$-cartesian if $d\ge3$ and $d-q_i\ge3$ for all $i$. The proof of this, carried out in this subsection, is organised as follows:
\begin{enumerate}[label=$\circled{\arabic*}$, ref={$\circled{\arabic*}$}]
	\item \label{step:scan} First we explain that it suffices to show that the composition 
	\[
		\rho=(\tau\circ\sigma)\colon \CE(\ast,M_\bullet)\lra \Omega\CE(\ast \times\{0\},M_\bullet\times J)
	\]
	is $(2d-4+\Sigma)$-cartesian, where $\tau$ is the scanning map from \cref{sec:delooping}. 
	\item \label{step:cimm} Next we reduce to proving the analogous statement for the analogous map
	\[
		\rho \colon \CbE(\ast,M_\bullet) \lra \Omega \CbE(\ast \times\{0\},M_\bullet \times J)
	\]	
	where $\CbE$ denotes the homotopy fibre of the forgetful map from concordance embeddings to concordance immersions.
	\item \label{step:partial} For the next step we consider the subspace $\CE^A(\ast ,M)\subset\CF(\ast ,M)$ consisting of those concordance maps that are embeddings on a fixed submanifold $A$ of $I$; similarly, we define $\CI^A(\ast ,M)$ and $\CbE^A(\ast ,M)$. Using these, we argue that it suffices to prove that the analogous map \[\rho\colon \CbE^A(\ast ,M_\bullet) \lra \Omega \CbE^A(\ast \times\{0\},M_\bullet \times J)\] is $(2d-1+\Sigma)$-cartesian when $A$ is the complement of three open intervals in $\interior(I)$. We argue further that for this purpose $\CbE^A(\ast ,-)$ may be replaced by $\CE^{\{t_1,t_2\}}(\ast ,-)$ where $\{t_1,t_2\} \subset \interior(I)$ is a two-element subset.
	\item \label{step:blakers-massey} We finish the proof by showing that the map of $r$-cubes
	\[
		\rho\colon \CE^{\{t_1,t_2\}}(\ast ,M_\bullet) \lra \Omega \CE^{\{t_1,t_2\}}(\ast \times\{0\},M_\bullet \times J)
	\]
	is $(2d-1+\Sigma)$-cartesian.
\end{enumerate}

\subsubsection*{Step \ref{step:scan}: Scanning}
The $1$-disc $D^1=\ast\times J\subset M\times J$ satisfies $D^1\cap \partial(M\times J)=\partial D^1$, so \cref{sec:delooping} gives a scanning map
\[
	\tau\colon \CE(\ast\times J,M  \times J)\lra \Omega \CE(\ast\times \{0\},M  \times J).
\]
In that section, this map was only defined up to contractible choices, but for what follows, it will be beneficial to fix a particular model. To do so, note that if $e$ is a concordance embedding of $\ast \times J$ into $M \times J$ then for each $s\in\interior(J)$, the restriction $e_s\coloneq e|_{\ast\times \{s\}\times I}$ is a concordance embedding of $\ast\times \{s\}$ into $M\times J$, which agrees with the inclusion for $s$ in a neighbourhood of $\partial J$. To adjust $e_s$ to make it a concordance embedding $\tau(e)_s$ of $\ast\times \{0\}$ into $M\times J$ (instead of $\ast\times \{s\}$), we fix once and for all a smooth family of diffeomorphisms $h_s \colon J\to J$ for $s\in \interior(J)$ that satisfies $h_0=\id_J$ and $h_s(s)=0$ for all $s$. Using this family of diffeomorphisms, we define
\[
	\tau(e)_s \coloneq \begin{cases} (\id_M \times h_s \times \id_I) \circ e_s \circ (\id_\ast \times h^{-1}_s \times \id_I) & \text{for $s \in \interior(J)$,} \\
	\inc_{\ast \times \{0\} \times I} & \text{for $s \in \partial J$.}\end{cases}
\]
The resulting map $J \ni s\mapsto \tau(e)_s \in \CE(\ast\times \{0\},M\times J)$ defines a loop $\tau(e) \in \Omega \CE(\ast\times \{0\},M\times J)=\Map_\partial(J,\CE(\ast\times \{0\},M\times J))$. This construction depends continuously on $e$, so it defines a map $\tau$ as desired.

\begin{lem}\label{lem:same-scanning}This map is homotopic to the map considered in \cref{sec:delooping} for $p=1$.
\end{lem}

\begin{proof}
Going through the construction in \cref{sec:delooping}, we see it suffices to show that the map $\tau$ defined above can be obtained in the following way: fix deformation retractions of $\CE(\ast \times [-1,0],M \times J)$ and $\CE(\ast \times [0,1],M \times J)$ onto the inclusion, and map $e\in \CE(\ast \times J,M \times J)$ to the loop in $\CE(\ast \times \{0\},M \times J)$ based at the inclusion obtained by concatenating the two paths from the restriction of $e$ to a concordance embedding of $\ast\times \{0\}\subset \ast\times J$ to the inclusion, resulting from restricting the two deformation retractions to $\ast \times \{0\}$. 

To show that $\tau$ is of this form, we consider for $e \in \CE(\ast  \times J,M \times J)$ and $s\in [-1,0]$ the family of concordance embeddings of $\ast\times [-1,0]$ into $M\times J\times I$ given by
\[
	\begin{cases} (\id_M \times h_{s} \times \id_I) \circ e \circ (\id_\ast \times h^{-1}_{s} \times \id_I) & \text{for $s\in (-1,0]$,} \\
	\inc_{\ast\times [-1,0]\times I} & \text{for $s=-1$.}\end{cases}
\]
Varying $s\in [-1,0]$, this defines a deformation retraction of $\CE(\ast  \times [-1,0],M \times J)$ onto the inclusion (for continuity at $s=-1$, use that $h_s^{-1}$ maps $[-1,0]$ into an arbitrary small neighbourhood of $-1$ as $s$ approaches $-1$). When we restrict it to a family of concordance embeddings of $\ast \times \{0\}$ into $M \times J$, it visibly agrees with the family $\tau(e)_s$ for $s\in[-1,0]$. Replacing $[-1,0]$ by $[0,1]$ defines a similar deformation retraction of $\CE(\ast  \times [0,1],M \times J)$. Using these two deformation retractions in the above discussion, the claim follows.
\end{proof}

The construction of $\tau$ above is natural in inclusions of codimension $0$ submanifolds $M\subset M'$ with $P\cap \partial M=P\cap \partial M'$ by postcomposition and hence extends to a map of $r$-cubes
\[
	\tau\colon \CE(\ast\times J,M _\bullet \times J)\lra \Omega \CE(\ast \times \{0\},M _\bullet \times J).
\]
This agrees with the scanning map considered in \cref{sec:delooping} up to homotopy of $r$-cubes, so it is $(2d-4+\Sigma)$-cartesian as an $(r+1)$-cube by an application of \cref{lem:delooping}. Thus, to show that the stabilisation $(r+1)$-cube \eqref{equ:stabilisation-point} is $(2d-5+\Sigma)$-cartesian, it suffices by \cref{lem:cubelemma} \ref{cubelemma:v} to show that the composition
\begin{equation*}\label{tausigma}
	\CE(\ast,M_\bullet )\xlra{\sigma} \CE(\ast\times J,M _\bullet \times J)\xlra{\tau} \Omega \CE(\ast\times \{0\},M _\bullet \times J)
\end{equation*}
is $(2d-5+\Sigma)$-cartesian. In fact, we will find that this composition is $(2d-4+\Sigma)$-cartesian.

\subsubsection*{Step \ref{step:cimm}: Fibre over immersions}
Recall from \cref{sec:immersions} that $\sigma$ and $\tau$ extend to maps
\begin{equation}\label{eq:CF-maps}
	\CF(\ast,M) \xlra{\sigma} \CF(\ast\times J,M \times J) \xlra{\tau} \Omega \CF(\ast \times \{0\},M  \times J).
\end{equation}
between the spaces of concordance \emph{maps}, which in turn restrict to analogous maps between the spaces of concordance \emph{immersions}. For immersions, $\tau$ is an equivalence by \cref{lem:cimm-scanning-equivalence} and $\sigma$ is $(2d-2)$-connected by \cref{lem:cimm-freudenthal}. In the multirelative setting
\[
	\CI(\ast,M_\bullet) \xlra{\sigma} \CI(\ast\times J,M_\bullet \times J) \xlra{\tau} \Omega \CI(\ast \times \{0\},M_\bullet  \times J),
\]
each of the three $r$-cubes involved is constant as a result of \cref{lem:cimm-independent-of-target}, so for $r>0$ they are $\infty$-cartesian by \cref{cor:cube-of-cubes} \ref{enum:constant-cubes}  which implies that the composed map $(\tau\circ\sigma)$ of $r$-cubes is $\infty$-cartesian for $r>0$. For $r=0$ the composed map is $(2d-2)$-cartesian, as noted above, and therefore in both cases the composed map of $r$-cubes is thus $(2d-2+\Sigma)$-cartesian.

Because of this, to show that $(\tau\circ\sigma)$ is $(2d-4+\Sigma)$-cartesian for concordance embeddings, it suffices by \cref{lem:cubelemma} \ref{cubelemma:i} and \ref{cubelemma:iii} to show that the composition of maps of $r$-cubes
\begin{equation}\label{equ:tausigma-bar}
	\CbE(\ast,M_\bullet) \xlra{\sigma} \CbE(\ast\times J,M _\bullet \times J) \xlra{\tau} \Omega \CbE(\ast \times \{0\},M _\bullet \times J)
\end{equation}
is $(2d-4+\Sigma)$-cartesian as an $(r+1)$-cube. Here 
\[
	\CbE(\ast,M)\coloneq \hofib_{\inc}(\CE(\ast,M)\ra\CI(\ast,M)).
\]
In this reduction, we implicitly used that $\CI(\ast,M_\varnothing)$ is connected (as a result of \cref{rem:immersion-point} \ref{enum:immersion-point-ii}; remember that $d\ge3$) which ensures that it suffices to consider fibres over the inclusion.

\begin{figure}
	\begin{tikzpicture}[scale=2.5]
		\draw (-2.5,0) -- (-2.5,1);
		\draw [ultra thick,Mahogany] (-2.5,0) -- (-2.5,.15);
		\draw [ultra thick,Mahogany] (-2.5,.95) -- (-2.5,1);
		\draw [ultra thick,Mahogany] (-2.5,.4) -- (-2.5,.5);
		\draw [ultra thick,Mahogany] (-2.5,.6) -- (-2.5,.75);		
		\node at (-2.7,.5) {$I$};
		\node at (-2.5,-.2) {$\ast$};
		
		\draw [->] (-2.25,.5) -- (-1.25,.5);
		\node at (-1.75,.6) {$e$};
		
		\draw (-1,0) -- (1,0) -- (1,1) -- (-1,1) -- cycle;
		\node at (1.2,.5) {$I$};
		\node at (0,-.2) {$M$};
		\draw (0,0) -- (0,.15) to[out=90,in=-90] (-.5,.3) to[out=90,in=110] (.7,.4) to[out=-70,in=45] (.6,.1) to[out=-135,in=-45] (.3,.6) to[out=135,in=0] (0,.8) to[out=180,in=180] (.1,.4) to[out=0,in=-135] (.5,.8) to[out=45,in=-90] (.8,.95) -- (.8,1);
		\draw [ultra thick,Mahogany] (0,0) -- (0,.15);
		\draw [ultra thick,Mahogany] (.8,.95) -- (.8,1);
		\draw [ultra thick,Mahogany] (.7,.4) to[out=-70,in=45] (.6,.1);
		\draw [ultra thick,Mahogany] (0,.8) to[out=180,in=180] (.1,.4);
	\end{tikzpicture}
	\caption{An element $e$ of $\CE^A(*,M)$. The compact submanifold $A \subset I$ is indicated in thick red.}
	\label{fig:cea-example}
	\end{figure}
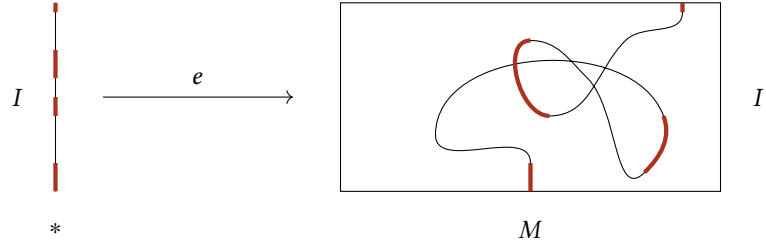

\subsubsection*{Step \ref{step:partial}: Partial embeddings and partial immersions} To prove that the composition \eqref{equ:tausigma-bar} is $(2d-4+\Sigma)$-cartesian, we consider further subspaces of the space $\CF(\ast,M)$ of concordance maps. Namely, for a compact submanifold $A\subset I$ (which may be $0$- or $1$-dimensional), we consider the subspaces
\begin{equation}\label{equ:partial-embeddings}
	\CE^A(\ast,M)\subset \CI^A(\ast,M)\subset \CF(\ast,M)
\end{equation}
consisting of those concordance maps $\ast\times I\ra M\times I$ whose restriction to $\ast\times A$ is an embedding (this defines $\CE^A(\ast,M)$) or an immersion (this defines $\CI^A(\ast,M)$); see \cref{fig:cea-example} for an example. Analogous to the definition of $\CbE(\ast,M)$, we write
\[
	\CbE^A(\ast,M)\coloneq\hofib_\inc(\CE^A(\ast,M)\ra \CI^A(\ast,M)).
\]

\begin{lem}\label{lem: tilt}The composition \eqref{eq:CF-maps} preserves the subspaces \eqref{equ:partial-embeddings} in that we have 
\[
	(\tau\circ\sigma)\big(\CE^A(\ast,M)\big)\subset \Omega \CE^A(\ast \times \{0\},M\times J)\quad\text{and}\quad(\tau\circ\sigma)\big(\CI^A(\ast,M)\big)\subset\Omega\CI^A(\ast \times \{0\},M\times J)
\]
for any compact submanifold $A\subset I$.
\end{lem}

\begin{proof}
Going through the definition, we see that for $f\in \CF(\ast,M)$ the value at $0\in J=[-1,1]$ of the loop $(\tau\circ\sigma)(f)\colon J\ra \CF(\ast \times \{0\},M  \times J)$ agrees with the composition $\ast\times I\ra M\times I=M\times \{0\}\times I\subset M\times J\times I$ of $f$ with the inclusion, so it is an embedding (or immersion) on $\ast\times A$ if this holds for $f$. At any $s\neq 0$ the value turns out to be an embedding on \emph{all} of $\ast\times I$. To show this, since  $\tau$ restricts to a map between spaces of concordance embeddings, it suffices to prove that for any concordance map $f\colon \ast\times I\ra M\times I$, the restriction of $\sigma(f)\colon \ast\times J\times I\ra M\times J\times I$ along $\ast\times \{s\}\subset \ast\times J$ is an embedding for all $s\neq0$.

By the construction of the stabilisation map in \cref{sec:stabilisation-map} (and using the notation from that section), we have $\sigma(f)^{-1}(M \times D_i) \subset \ast \times D_i$ for $i=1,2$, and near $\ast \times D_2$ the map $\sigma(f)$ is the inclusion. Hence it suffices to prove that the map $g \coloneqq \big(\pr_{J \times I} \circ \sigma(f)|_{(\{s\} \times I) \cap D_1}\big) \colon(\{s\} \times I) \cap D_1 \ra J \times I$,
is an embedding. Noting that the restriction of the parametrisation $\Lambda$ from \eqref{equ:polar-coordinates} to a map $\Lambda' \colon [0,1) \times [0,\pi] \to D_1 \backslash \{(0,1)\}$ is a diffeomorphism and $\im(g) \subset \im(\Lambda')$, it suffices to prove that the composition $(\pr_{[0,\pi]} \circ \Lambda'^{-1} \circ g)$ is an embedding. Tracing through the definition and identifying $(\{s\} \times I) \cap D_1$ with $\smash{[1-\sqrt{1-s^2},1]}$ via $\pi_2 \colon J \times I \to I$, one sees that this composition is given by the formula $\smash{[1-\sqrt{1-s^2},1]\ni t\mapsto \arctan(-(1-t)/s)\in [0,\pi]}$ which is indeed an embedding. See \cref{fig:stabilisation-rest} for an illustration.
\end{proof}

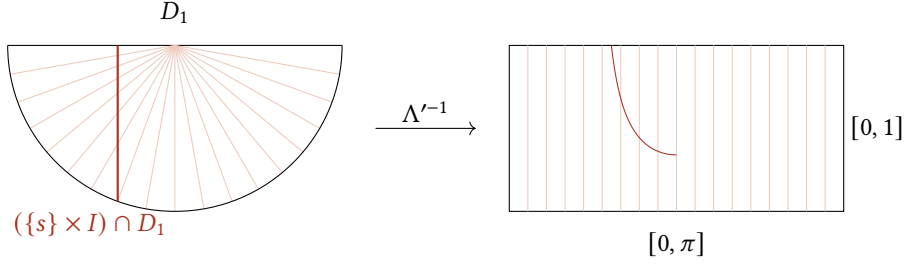
\begin{figure}
	\begin{tikzpicture}[scale=2.2]
		\draw (1,1) -- (-1,1);
		\foreach \i in {1,...,17}
		{
			\draw [Mahogany!20!white] (0,1) ++({-\i*10}:0) -- ++({-\i*10}:1);
		}
		\draw (1,1) arc [start angle=0, end angle=-180, x radius=1cm,y radius =1cm];
		
		\node at (0,1.2) {$D_1$};
		\draw [thick,Mahogany] ({cos(110)},{1-sin(110)}) -- ({cos(110)},1);
		
		\draw [->] (1.2,.5) -- (1.8,.5);
		\node at (1.5,.6) {$\Lambda'^{-1}$};
		
		\node [Mahogany] at ({1.2*cos(110)-.1},{1-1.2*sin(110)+.05}) {$(\{s\} \times I) \cap D_1$};
		
		\begin{scope}[xshift=3cm]
			\draw (1,1) rectangle (-1,0);
			\foreach \i in {1,...,17}
			{
				\draw [Mahogany!20!white] ({-1+2*\i/18},1) -- ({-1+2*\i/18},0);
			}
		\draw[domain=0.06:1,smooth,variable=\x,Mahogany] plot ({atan(-(1-\x)/0.34)/180},{sqrt(0.34^2+(1-\x)^2)});
		\node at (0,-.2) {$[0,\pi]$};
		\node at (1.2,.5) {$[0,1]$};
		\end{scope}
	\end{tikzpicture}
	\caption{The restriction of $\pr_{J \times I} \circ \sigma(f)$ to the indicated interval changes the radial coordinate (in $[0,1]$) but not the angle (in $[0,\pi]$). In particular, the compositions with $\pr_{[0,\pi]} \circ \Lambda'^{-1} \circ \pr_{J \times I} $ of $\sigma(f)$ and $\sigma(\inc)$ agree.}
	\label{fig:stabilisation-rest}
\end{figure}

As a result of \cref{lem: tilt}, we have induced maps 
\[
	\CE^A(\ast,M)\lra \Omega\CE^A(\ast \times \{0\},M\times J)\quad\text{and}\quad
	\CI^A(\ast,M) \lra \Omega\CI^A(\ast \times \{0\},M\times J)
\]
and thus also a map $\CbE^A(\ast,M) \ra \Omega\CbE^A(\ast \times \{0\},M\times J)$ on homotopy fibres over the inclusion. We denote all of these by $\rho$. Note that they are natural in inclusions of submanifolds $A'\subset A$ and $M\subset M'$ with $\partial M\cap P=\partial M'\cap P$. In particular, choosing three disjoint closed intervals $B_1,B_2,B_3\subset\interior(I)$ in the order $B_1<B_2<B_3$, and writing $A_T\coloneq \closure(I\backslash \cup_{i\in S} B_i)$,
for $T\subseteq\ul{3}=\{1,2,3\}$, the maps $\rho\colon \CbE^{A_{T}}(\ast,M_{S}) \ra \Omega \CbE^{A_{T}}(\ast \times \{0\},M_{S}\times J)$ assemble to a map 
\begin{equation}\label{eq: map of 3+r cubes}
	\rho\colon \CbE^{A_{\smallsquare}}(\ast,M_\bullet) \lra \Omega \CbE^{A_{\smallsquare}}(\ast \times \{0\},M_\bullet\times J).
\end{equation}
of $(3+r)$-cubes. This map induces a commutative square of $r$-cubes 
\begin{equation}\label{square with rho}
	\begin{tikzcd}\CbE^{A_\varnothing}(\ast,M_\bullet) \rar{\rho} \dar & \Omega \CbE^{A_\varnothing}(\ast \times \{0\},M_\bullet\times J) \dar \\
	\underset{\varnothing\neq T\subseteq\ul{3}}\holim\, \CbE^{A_T}(\ast, M_\bullet) \rar{\rho} & \underset{\varnothing\neq T\subseteq\ul{3}}\holim\, \Omega \CbE^{A_T}(\ast \times \{0\},M_\bullet\times J) \end{tikzcd}
\end{equation}
whose top map is the composition \eqref{equ:tausigma-bar} from Step \ref{step:cimm} (note that $A_\varnothing = I$), which we wish to prove to be $(2d-4+\Sigma)$-cartesian. It will follow from the next lemma (applied to $M$ and $M\times J$) that the left vertical map is $(2d-4+\Sigma)$-cartesian and the right vertical map is $(2d-3+\Sigma)$-cartesian, so using \cref{lem:cubelemma} \ref{cubelemma:iv} and \ref{cubelemma:v} it will be enough to show that the bottom map is $(2d-4+\Sigma)$-cartesian.

\begin{lem}\label{lem: three holes}
The $(3+r)$-cube $\CbE^{A_{\smallsquare}}(\ast,M_\bullet)$ is $(2d-4+\Sigma)$-cartesian.
\end{lem}

\begin{proof}We will show that $\CE^{A_{\smallsquare}}(\ast ,M_\bullet)$ is $(2d-4+\Sigma)$-cartesian and that $\CI^{A_\smallsquare}(\ast ,M_\bullet)$ is $\infty$-cartesian, which will imply the claim by combining \cref{lem:cubelemma} \ref{cubelemma:ii} and \ref{cubelemma:iii}. We begin with $\CE^{A_{\smallsquare}}(\ast ,M_\bullet)$. Adopting the notation from \cref{sec:multiple-disjunction}, we have a map of $(3+r)$-cubes $\CE^{A_{\smallsquare}}(\ast,M_\bullet) \ra \E(A_{\smallsquare},M_\bullet \times I)$ by restriction.
The $(3+r)$-cube $\E(A_{\smallsquare},M_\bullet \times I)$ is $(2d-4+\Sigma)$-cartesian by \cref{cor:disjunction-excision} (note that $d+1-1\ge3$ since $d+1\ge4$) and we will show next that the $(3+r)$-cubes of homotopy fibres over all basepoints are $\infty$-cartesian, so the claim will follow by \cref{lem:cubelemma} \ref{cubelemma:i} and \ref{cubelemma:iii}. These $(3+r)$-cubes of homotopy fibres have the form
\begin{equation*}\label{equ:cutting-cube-fibres}
	\textstyle{\ul{3} \times \ul{r} \supseteq T \times S \longmapsto \bigsqcap_{i \in T} \Map_\partial(B_i,M_S \times I)}
\end{equation*}
where the boundary conditions in the mapping spaces $\Map_\partial(B_i,M_S \times I)$ depend on the basepoint in $\E(I,M_\varnothing \times I)$ one takes homotopy fibres over. Combining \cref{lem:cube-of-cubes} \ref{cube-of-cubes:ii} and \cref{cor:cube-of-cubes} \ref{enum:pointed-product cube}, we see that these $(3+r)$-cubes are $\infty$-cartesian, as claimed.

The claim that $\CI^{A_{\smallsquare}}(\ast ,M_\bullet)$ is $\infty$-cartesian can be proved similarly: we have a restriction map $\CI^{A_{\smallsquare}}(\ast,M_\bullet) \ra \I(A_{\smallsquare},M_\bullet \times I)$
whose target is the analogue of $\E(A_\smallsquare,M_\bullet \times I)$ for immersions. The $(3+r)$-cubes of homotopy fibres are of the same form as previously and thus $\infty$-cartesian, so it suffices to show that $\I(A_{\smallsquare},M_\bullet \times I)$ is $\infty$-cartesian. To see this, we consider the restriction map $\I(A_\smallsquare,M_\bullet) \to \I(A_{\ul{3}},M_\bullet)$. The target $(3+r)$-cube is constant in some directions, so it is $\infty$-cartesian, and the $(3+r)$-cubes of homotopy fibres are all of the form treated in \cref{cor:cube-of-cubes} \ref{enum:product-cube}, so they are also $\infty$-cartesian.
\end{proof}

We are left to show that the bottom map in the diagram \eqref{square with rho} of $r$-cubes is $(2d-4+\Sigma)$-cartesian when considered as an $(r+1)$-cube; in fact we will find that it is $(2d-3+\Sigma)$-cartesian. The proof of this relies on the following lemma on $\CbE^A(\ast,M)$.

\begin{lem}\label{lem: points} Let $A\subset I$ be a compact $1$-dimensional submanifold that contains $\partial I$ and has $k+2$ path components for some $k\ge0$. Choose points $\{t_1,\dots ,t_k\}\subset A$, one in each path component of $A$ in the interior of $I$. Then the composition
\[
	\CbE^A(\ast,M)\lra \CE^A(\ast,M)\overset{\subset}\lra \CE^{\{t_1,\dots ,t_k\}}(\ast,M)
\]
is an equivalence. In particular, $\CbE^A(\ast,M)$ is contractible for $k \le1$.
\end{lem}

\begin{proof}Using that the forgetful map $\E([0,1],N)\ra \I([0,1],N)$ is an equivalence for any manifold $N$, one sees that the commutative square of inclusions
\[
	\begin{tikzcd} \CE^A(\ast ,M) \rar \dar & \CE^{\{t_1,\dots ,t_k\}}(\ast ,M) \dar \\
	\CI^A(\ast ,M) \rar& \CI^{\{t_1,\dots ,t_k\}}(\ast ,M)\end{tikzcd}
\]
is homotopy cartesian. Taking vertical homotopy fibres over the inclusion and using that $\CI^{\{t_1,\dots ,t_k\}}(\ast ,M)=\CF(\ast ,M)$ is contractible by \cref{lem:CF-contractible}, this implies the claim. The addendum follows by noting that for $k\le1$, we have $\CE^{\{t_1,\dots ,t_k\}}(\ast,M)=\CF(\ast,M)$.\end{proof}

Choosing a point $t_1\in \interior(I)$ between $B_1$ and $B_2$ and a point $t_2\in \interior(I)$ between $B_2$ and $B_3$ (so in particular $t_1<t_2$), we obtain equivalences of the form
\[
	\underset{\varnothing\neq T\subseteq\ul{3}}\holim\, \CbE^{A_T}(\ast,M)\xlla{\simeq} \Omega^2 \CbE^{A_{\ul{3}} }(\ast,M)\xlra{\simeq} \Omega^2 \CE^{\{t_1,t_2\}}(\ast,M),
\]
where the left equivalence uses that the inclusion of the basepoint into $\CbE^{A_T}(\ast,M)$ is an equivalence whenever $\varnothing \neq S\subsetneq \ul{3}$, by \cref{lem: points}, and the right equivalence uses that the canonical map $\CbE^{A_{\ul{3}} }(\ast,M)\ra \CE^{\{t_1,t_2\}}(\ast,M)$ is an equivalence, by the same lemma. Replacing $M$ by $M\times J$, there are the analogous equivalences
\[
	\underset{\varnothing\neq T\subseteq\ul{3}}\holim\, \CbE^{A_T}(\ast \times \{0\},M\times J) \xlla{\simeq} \Omega^2 \CbE^{A_{\ul{3}} }(\ast \times \{0\},M\times J) \xlra{\simeq} \Omega^2 \CE^{\{t_1,t_2\}}(\ast \times \{0\},M\times J).
\]
All these equivalences are natural in $M$ and compatible with $\rho$, so we may extend diagram \eqref{square with rho}, obtaining a commutative diagram
\begin{equation} \label{square with rho2}
	\begin{tikzcd}[ar symbol/.style = {draw=none,"\textstyle#1" description,sloped},
  equivalence/.style = {ar symbol={\simeq}}]
	\CbE(\ast,M_\bullet) \rar{\rho} \dar & \Omega \CbE(\ast \times \{0\},M_\bullet\times J) \dar \\
	\underset{\varnothing\neq T\subseteq\ul{3}}\holim\, \CbE^{A_T}(\ast, M_\bullet) \rar{\rho} \arrow[d,equivalence,start anchor={[xshift=-5pt]}]& \underset{\varnothing\neq T\subseteq\ul{3}}\holim\, \Omega \CbE^{A_T}(\ast \times \{0\},M_\bullet\times J)\arrow[d,equivalence,start anchor={[xshift=-5pt]}]\\[-0.3cm]
	\Omega^2\CE^{\{t_1,t_2\}}(\ast, M_\bullet )\rar{\Omega^2\rho}&\Omega^3 \CE^{\{t_1,t_2\}}(\ast \times \{0\},M_\bullet \times J)
	\end{tikzcd}
\end{equation}
where the vertical $\simeq$-signs indicate a zig-zag of equivalences compatible with the horizontal maps. Therefore, to show that the bottom map of $r$-cubes in \eqref{square with rho} is $(2d-3+\Sigma)$-cartesian (and thus also the top row), it suffices to show that the map of $r$-cubes
\begin{equation}\label {rho}
	\rho\colon \CE^{\{t_1,t_2\}}(\ast, M_\bullet ) \lra \Omega \CE^{\{t_1,t_2\}}(\ast \times \{0\},M_\bullet \times J)
\end{equation}
is $(2d-1+\Sigma)$-cartesian as an $(r+1)$-cube.

\subsubsection*{Step \ref{step:blakers-massey}: Applying the Blakers--Massey theorem}In this step, we finish the proof by showing that \eqref{rho} is $(2d-1+\Sigma)$-cartesian. To this end, we first make two alterations to the map of $r$-cubes \eqref{rho}: enlarging its target by an equivalence and performing a homotopy. Some of the ideas involved are similar to those in the proof of \cref{lem:cimm-freudenthal}.

The enlargement of the target is done by considering the two subspaces
\[
	C_+(\ast, M)\subset \CE^{\{t_1,t_2\}}(\ast \times \{0\},M\times J)\quad\text{and}\quad C_-(\ast, M)\subset \CE^{\{t_1,t_2\}}(\ast \times \{0\},M\times J)
\]
where $C_+(\ast, M)\subset \CE^{\{t_1,t_2\}}(\ast \times \{0\},M\times J)$ consists of those concordance maps $f\colon \ast \times \{0\}\times I\ra M\times J\times I$ for which $f(\ast,0,t_1)$ is \emph{not} directly to the right of $f(\ast,0,t_2)$, where $(x_1,s_1,r_1)$ in $M\times J\times I$ is \emph{directly to the right of $(x_2,s_2,r_2)$} if $x_1=x_2$, $r_1=r_2$, and $s_1>s_2$.  The space $C_-(\ast, M)$ is defined similarly, replacing right by left. In terms of these subspaces, we define
\[
	\Omega^{\pm} \CE^{\{t_1,t_2\}}(\ast \times \{0\},M \times J)
\]
as the space of paths in $ \CE^{\{t_1,t_2\}}(\ast \times \{0\},M \times J)$ that start in $C_+(\ast,M)$ and end in $C_-(\ast,M)$, i.e.\,the homotopy limit of the zig-zag
\[
	C_+(\ast, M) \xlra{\subset} \CE^{\{t_1,t_2\}}(\ast \times \{0\},M\times J)\xlla{\supset} C_-(\ast, M).
\]
Including the basepoint into $C_+(\ast, M)$ and  $C_-(\ast, M)$ induces an inclusion
\begin{equation}\label{equ:thicken-source}
	\Omega \CE^{\{t_1,t_2\}}(\ast \times \{0\},M \times J)\xlra{\subset} \Omega^{\pm} \CE^{\{t_1,t_2\}}(\ast \times \{0\},M \times J)
\end{equation}
which is an equivalence as a result of the following lemma.

\begin{lem}$C_+(\ast, M)$ and $C_-(\ast, M)$ are contractible. In particular, \eqref{equ:thicken-source} is an equivalence.\end{lem}

\begin{proof}
It suffices to prove that $\CF(\ast \times \{0\}, M\times J)$ deformation-retracts onto the subspace $C_+(\ast,M)$ (respectively $C_-(\ast,M)$), since $\CF(\ast \times \{0\}, M\times J)$ is contractible by \cref{lem:CF-contractible}. To see this one deforms $f\in \CF(\ast \times \{0\}, M\times J)$ in such a way that $f(t_1)$ moves directly to the left (respectively right) while $f(t_2)$ moves directly to the right (respectively left). We leave it to the reader to provide an explicit formula. 
\end{proof}

Note that the equivalence \eqref{equ:thicken-source} is natural in $M$, so we may replace \eqref{rho} by the composition
\[
	\rho^{\pm}\colon \CE^{\{t_1,t_2\}}(\ast, M_\bullet ) \xlra{\rho}\Omega \CE^{\{t_1,t_2\}}(\ast \times \{0\},M_\bullet \times J) \xlra{\subset}\Omega^{\pm} \CE^{\{t_1,t_2\}}(\ast \times \{0\},M_\bullet \times J).
\]
The proof will be completed by showing that the map $\rho^{\pm}$ of $r$-cubes is $(2d-1+\Sigma)$-cartesian as an $(r+1)$-cube. To do so, we first note that the subspaces $C_+(\ast, M)$ and $C_-(\ast, M)$ of $\CE^{\{t_1,t_2\}}(\ast \times \{0\},M\times J)$ are open, and that their union is the entire space. Note also that their intersection is equivalent to $\CE^{\{t_1,t_2\}}(\ast,M)$, since viewing a map $f\colon \ast\times\{0\}\times I\ra M\times J\times I$ as a pair of a map $\ast\times I\ra M\times I$ and a map $I\ra J$ induces a homeomorphism from $C_+(\ast, M)\cap C_-(\ast, M)$ to the product of $\CE^{\{t_1,t_2\}}(\ast,M)$ with the contractible space of smooth maps $I\to J$ that take a neighbourhood of $0$ to $0$. Identifying the subspace  $\CE^{\{t_1,t_2\}}(\ast \times \{0\},M\times \{0\})\subset\CE^{\{t_1,t_2\}}(\ast \times \{0\},M\times J)$ with $\CE^{\{t_1,t_2\}}(\ast ,M)$ we thus have a homotopy cocartesian square 
\begin{equation}\label{eq: pushout square}
	\begin{tikzcd}
	\CE^{\{t_1,t_2\}}(\ast,M)\rar\dar & C_+(\ast,M)\dar\\
	C_-(\ast,M)\rar & \CE^{\{t_1,t_2\}}(\ast \times \{0\},M\times J).
	\end{tikzcd}
\end{equation}

\begin{lem}\label{lem: left-right}
For $f\in \CE^{\{t_1,t_2\}}(\ast,M)$, the loop 
\[
	\rho(f)\colon J \lra \CE^{\{t_1,t_2\}}(\ast \times \{0\},M \times J)
\] 
satisfies $\rho(f)_s\in C_+(\ast, M)$ for $s\in[-1,0]\subset J$ and $\rho(f)_s\in C_-(\ast, M)$ for $s\in[0,1]\subset J$
\end{lem}

\begin{proof}
Since $\rho(f)_0$ agrees with the composition of $f$ with the inclusion $M \times I=M\times\{0\}\times I\subset M \times J \times I$, its composition with $\pr_{M,I}$ is injective on $\{t_1,t_2\}$, so $\rho(f)_0\in C_+(\ast, M)\cap C_-(\ast, M)$. We may thus assume $s \neq 0$. Going through the definition, we see that $\rho(f)_s$ is obtained from $\sigma(f)_s= \sigma(f)|_{\ast \times \{s\} \times I}\colon \ast \times \{s\} \times I\ra \ast\times J\times M$ by postcomposition with the diffeomorphism $\id_M \times h_s \times \id_I$ from Step \ref{step:scan}. The latter preserves the property that the value at $t_1$ is not directly to the left (respectively right) of the value at $t_2$, so it suffices to prove that $\sigma(f)_s \in C_+(\ast,M)$ for $s < 0$ and $\sigma(f)_s \in C_-(\ast,M)$ for $s > 0$. The proof in the two cases are analogous. We will explain the former, so fix $s<0$.
	
We write $\sigma(f)(\ast,s,t_1) = (x_1,s_1,r_1)$ and $\sigma(f)(\ast,s,t_2)= (x_2,s_2,r_2)$, and encourage the reader to (a) recall the construction of $\rho$ in \cref{sec:stabilisation-map}, including the decomposition $D_1\cup D_2=J\times I$, and to (b) look at \cref{fig:stabilisation-left-right}. Intersecting the decomposition $D_1\cup D_2=J\times I$ with $\{s\} \times I$ gives a decomposition $I= (\{s\} \times I\cap D_1)\cup (\{s\} \times I\cap D_2)$. The map $\sigma(f)_s$ preserves this decomposition and agrees with the inclusion on $(\{s\} \times I\cap D_2)$, so it follows that if $t_1$ and $t_2$ do not both lie in $(\{s\} \times I\cap D_1)$, then $\sigma(f)(\ast,s,t_1)$ is neither directly to the right nor the left of $\sigma(f)(\ast,s,t_2)$. Otherwise, since $(\pr_{J,I} \circ \sigma(f))|_{\ast\times D_1}$ preserves the radial segments $\Lambda([0,1] \times \{\theta\}) \subset D_1$ for $\theta\in [0,\pi]$ where $\Lambda$ is as in \eqref{equ:polar-coordinates}, we have points $(\pr_{J,I} \circ \sigma(f))(\ast,s,t_1)$ and $(\pr_{J,I} \circ \sigma)(f)(\ast,s,t_2)$ must lie on different radial segments, where the latter is closer to $\Lambda([0,1] \times \{\pi/2\})=[-1,0] \times \{1\} \subset J \times I$. Then the only way to have $r_1 = r_2$ is if $s_1<s_2$, so $\sigma(f)(\ast,s,t_1)$ is not directly to the right of $\sigma(f)(\ast,s,t_2)$.
\end{proof}

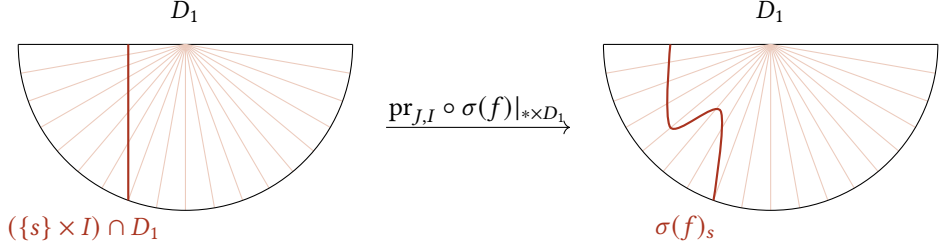
\begin{figure}
	\begin{tikzpicture}[scale=2.2]
		\draw (1,1) -- (-1,1);
		\foreach \i in {1,...,17}
		{
			\draw [Mahogany!20!white] (0,1) ++({-\i*10}:0) -- ++({-\i*10}:1);
		}
		\draw (1,1) arc [start angle=0, end angle=-180, x radius=1cm,y radius =1cm];
		
		\node at (0,1.2) {$D_1$};
		\draw [thick,Mahogany] ({cos(110)},{1-sin(110)}) -- ({cos(110)},1);
		
		\draw [->] (1.2,.5) -- (2.3,.5);
		\node at (1.75,.6) {$\pr_{J,I} \circ \sigma(f)|_{* \times D_1}$};
		
		\node [Mahogany] at ({1.2*cos(110)-.2},{1-1.2*sin(110)+.02}) {$(\{s\} \times I) \cap D_1$};
		
		\begin{scope}[xshift=3.5cm]
			\draw (1,1) -- (-1,1);
			\foreach \i in {1,...,17}
			{
				\draw [Mahogany!20!white] (0,1) ++({-\i*10}:0) -- ++({-\i*10}:1);
			}
			\draw (1,1) arc [start angle=0, end angle=-180, x radius=1cm,y radius =1cm];
			
			\node at (0,1.2) {$D_1$};
			\draw [thick,Mahogany] plot [smooth] coordinates {({-cos(70)},{1-sin(70)}) (-.3,.6) (-.6,.5) (-.6,1)};
			\node [Mahogany] at ({-1.2*cos(70)-.1},{1-1.2*sin(70)+.02}) {$\sigma(f)_s$};
		\end{scope}
	\end{tikzpicture}
	\caption{As $\sigma(f)$ preserves the radial segments indicated in light red, for $s<0$, if $\sigma(f)(\ast,s,t_1)$ has the same $I$-coordinate (depicted vertically) as $\sigma(f)(\ast,s,t_2)$ then it has smaller $J$-coordinate (depicted horizontally).}
	\label{fig:stabilisation-left-right}
\end{figure}

In view of \cref{lem: left-right}, we have a homotopy
\[
	[0,1]\times \CE^{\{t_1,t_2\}}(\ast,M)\lra\Omega^{\pm }\CE^{\{t_1,t_2\}}(\ast \times \{0\},M \times J).
\]
that sends $(u,f)$ to $[-1,1]\ni s\mapsto\rho(f)_{(1-u)s} \in \CE^{\{t_1,t_2\}}(\ast \times \{0\},M \times J)$ for $0\le u\le 1$. It starts at $\rho^{\pm}$ and ends at the map taking $f$ to the constant path at $\rho(f)_0$. The latter is the map induced by the commutative square \eqref{eq: pushout square} by mapping the upper left corner to the homotopy limit of the others. Since the homotopy is natural in $M$, this reduces the claim that $\rho^{\pm}$ is $(2d-1+\Sigma)$-cartesian to showing that the square of $r$-cubes \begin{equation}\label{eq: square of cubes}
	\begin{tikzcd}
	C_+(\ast,M_\bullet)\cap C_-(\ast,M_\bullet)\rar\dar & C_+(\ast,M_\bullet)\dar\\
	C_-(\ast,M_\bullet)\rar & \CE^{\{t_1,t_2\}}(\ast \times \{0\},M_\bullet\times J).
	\end{tikzcd}
\end{equation}
is $(2d-1+\Sigma)$-cartesian when considered as an $(r+2)$-cube. We intend to do so by means of the multirelative Blakers--Massey \cref{thm:blakers-massey}. However, the $(r+2)$-cube \eqref{eq: square of cubes} is not strongly cocartesian, so we will replace it---in two steps---by an $(r+2)$-cube that is. 

First, we consider the configuration space $\Conf(2,M\times J\times I)$ of ordered pairs $(p_1,p_2)$ of distinct points in the interior of $M\times J\times I$ and the fibration 
\[
\CE^{\{t_1,t_2\}}(\ast \times \{0\},M\times J)\lra \Conf(2,M\times J\times I)
\]
given by evaluation at $t_1$ and $t_2$. The sets $C_+(\ast,M)$ and $C_-(\ast,M)$, are the preimages of open sets $\Conf_+(2,M\times J\times I)$ and $\Conf_-(2,M\times J\times I)$, defined respectively by requiring $p_1$ to be not directly to the right and not directly to the left of $p_2$. This is natural in $M$, and thus implies that we have an $\infty$-cartesian map from (\ref{eq: square of cubes}) to the square of $r$-cubes 
\begin{equation}\label{middle cube}\begin{tikzcd}
\Conf_+(2,M_\bullet\times J\times I)\cap \Conf_-(2,M_\bullet\times J\times I)\rar\dar & \Conf_+(2,M_\bullet\times J\times I)\dar\\
\Conf_-(2,M_\bullet\times J\times I)\rar & \Conf(2,M_\bullet\times J\times I),
\end{tikzcd}\end{equation}
so it suffices to prove that this $(r+2)$-cube is $(2d-1+\Sigma)$-cartesian, using \cref{lem:cubelemma} \ref{cubelemma:i}. That the map from (\ref{eq: square of cubes}) to (\ref{middle cube}) is indeed $\infty$-cartesian follows from an application of \cref{lem:cubelemma} \ref{cubelemma:iii} and \cref{cor:cube-of-cubes} \ref{cube-of-cubes:i}, after noting that all fibres are equivalent.

Second, we view $\Conf(2,M_S\times J\times I)$ as a bundle over the interior of $M_S\times J\times I$ by mapping the pair $(p_1,p_2)$ to $p_1$, with subbundles given by the open subsets $\Conf_+(2,M_S\times J\times I)$ and $\Conf_-(2,M_S\times J\times I)$ as well as $\Conf_+(2,M_S\times J\times I)\cap \Conf_-(2,M_S\times J\times I)$. Replacing all of these by their fibres over the point $\ast \times \{0\}\times \{t_1\}$, we obtain a square of $r$-cubes
\begin{equation}\label{eq: last cube}\begin{tikzcd}
 \interior(M_\bullet\times J\times I)\backslash (\ast\times J\times \{t_1\})\rar\dar &  \interior(M_\bullet\times J\times I)\backslash (\ast\times (-1,0]\times \{t_1\})\dar\\
  \interior(M_\bullet\times J\times I)\backslash (\ast\times [0,1)\times \{t_1\})\rar &  \interior(M_\bullet\times J\times I)\backslash (\ast \times \{0\}\times \{t_1\}) ,
\end{tikzcd}\end{equation}
mapping to (\ref{middle cube}) by an $\infty$-cartesian map, so it suffices to show that \eqref{eq: last cube} is $(2d-1+\Sigma)$-cartesian by \cref{lem:cubelemma} \ref{cubelemma:iii}. That the map from (\ref{eq: last cube}) to (\ref{middle cube}) is indeed $\infty$-cartesian follows as for the map from (\ref{eq: square of cubes}) to (\ref{middle cube}); this time using that all base spaces are the same. 

To see that it is, we use the multirelative Blakers--Massey \cref{thm:blakers-massey}. The $(r+2)$-cube \eqref{eq: last cube} is strongly cocartesian because it is made by cutting out $(r+2)$ pairwise disjoint submanifolds from $\interior(M\times J\times I) \backslash (\ast \times \{0\}\times t_1)$ that are closed as subspaces. Two of these, $\ast\times (-1,0)\times \{t_1\}$ and  $\ast\times (0,1)\times \{t_1\}$, are $1$-dimensional, so the inclusions of their complements are $d$-connected by general position. The other $r$ inclusions are up to equivalence of the form $N\backslash R \hookrightarrow N$ for a submanifold $R$ with handle dimension $q_i+2$ relative to $R\cap \partial N$ (set $N=(M_{\varnothing}\times J\times I)\backslash(\ast\times J\times\{t_1\})$ and $R=Q_i\times J\times I$), so they are $(d-q_i-1)$-connected, again by general position. \cref{thm:blakers-massey} thus gives the degree of cartesianness of the $(r+2)$-cube \eqref{eq: last cube} as \[\textstyle{(1-(r+2)+d+d+\sum_{i=1}^{r}(d-q_j-1))=(2d-1+\sum_{i=1}^{r}(d-q_j-2))}\]
as desired.

\subsection{Digression}\label{sec:calculus}

We pause the proof of \cref{thm:main-relative} for a moment to comment on aspects of the proof of the case of a point from the previous subsection. None of them are necessary for the proof, so this subsection may be skipped on first reading.

\subsubsection{The space $\CE^{\{t_1,t_2\}}(\ast, M)$}The final two steps in the proof of the case of a point featured the space $\CE^{\{t_1,t_2\}}(\ast, M)$ of concordance maps $I\ra M\times I$ that are injective on $\{t_1,t_2\}$. Although it is not necessary for the proofs, it is worth pointing out the homotopy type of this space:

\begin{lem}\label{lem:SOmega}
There is an equivalence
\[\CE^{\{t_1,t_2\}}(\ast, M)\simeq S^{T_\ast M}\wedge\Omega_\ast (M)_+\]
where $\Omega_\ast (M)$ is the space of loops in $M$ based at $\ast\in M$, the subscript $+$ adds an disjoint base point, and $S^{T_\ast M}\cong S^d$ is the one-point compactification of the tangent space at $\ast\in M$.
\end{lem}

\begin{proof} It suffices to produce an equivalence between $\CE^{\{t_1,t_2\}}(\ast,M)$ and the homotopy fibre at $\ast\in M$ of the canonical retraction $S^{T_\ast M}\vee M\to M$. To do so, we first recall that $\CE^{\{t_1,t_2\}}(\ast,M)$ for $t_1<t_2$ in the interior of $I$ is the space of smooth maps $f \colon I\to M\times I$ with $f(t)=(\ast,t)$ in a neighbourhood of $0$, $f(1)\in M\times 1$, and $f(t_1)\neq f(t_2)$. We now perform a sequence of alterations to $\CE^{\{t_1,t_2\}}(\ast,M)$ without affecting its homotopy type. Firstly, we may replace ``smooth'' by ``continuous'' in the definition. Secondly, we consider the restriction map
\[
\CE^{\{ t_1,t_2\}}(\ast ,M)\lra \CE^{\{ t_1,t_2\}}(\ast ,M)_{t_2},
\]
where $\CE^{\{ t_1,t_2\}}(\ast ,M)_{t_2}$ is the space of maps $f\colon [0,t_2]\to M\times [0,1)$ such that $f(t)=(\ast,t)$ for all $t$ in a neighbourhood of $0$ and such that $f(t_1)\neq f(t_2)$. This is a fibration whose fibres are contractible, therefore an equivalence. Thirdly, we use the restriction map
\[
\CE^{\{ t_1,t_2\}}(\ast ,M)_{t_2}\lra \CE^{\{ t_1,t_2\}}(\ast ,M)_{t_1} ,
\]
where $\CE^{\{ t_1,t_2\}}(\ast ,M)_{t_1}$ is the space of maps $f \colon [0,t_1]\to M\times [0,1)$ such that $f(t)=(\ast,t)$ for all $t$ in a neighborhood of $0$. This is a fibration with contractible base, so its fibre over the map $t\mapsto (\ast,t)$ is equivalent to $\CE^{\{ t_1,t_2\}}(\ast ,M)_{t_2}$. This fibre agrees with the space of all paths in $M\times [0,1)$ starting at $(\ast,t_1)$ and not ending at $(\ast,t_1)$. This is the homotopy fibre at $(\ast,t_1)$ of the inclusion 
$(M\times [0,1))\backslash \{(\ast,t_1)\}\subset M\times [0,1)$ which is in turn equivalent to the homotopy fibre at $\ast$ of the retraction $S^{T_\ast M}\vee M\to M$.
\end{proof}

\subsubsection{A variation of Step \ref{step:blakers-massey}}
There is an alternative way to carry out the final Step \ref{step:blakers-massey} in the proof of \cref{thm:main-relative} for a point. It is based on the observation that the equivalence produced in the proof of \cref{lem:SOmega} is natural in codimension $0$ embeddings of $M$, so it extends to an equivalence of cubes $\CE^{\{t_1,t_2\}}(\ast, M_\bullet)\simeq S^{T_\ast M}\wedge\Omega_\ast (M_\bullet)_+$. Up to this equivalence, the final square \eqref{square with rho2} in Step \ref{step:partial} may thus be written as
\begin{equation}\label{new square with rho}
	\begin{tikzcd}\CbE(\ast,M_\bullet) \rar{\rho} \dar[swap]{\pi} & \Omega \CbE (\ast \times \{0\},M_\bullet\times J) \dar{\Omega\pi} \\
	\Omega^2\Sigma^d\Omega(M_\bullet)_+ \rar & \Omega^3\Sigma^{d+1}\Omega(M_\bullet)_+. \end{tikzcd}
\end{equation}
In these terms, the task in Step \ref{step:blakers-massey} was to show that the lower horizontal map in the square is $(2d-3+\Sigma)$-cartesian. The description of the bottom entries of this square suggests the following strategy: first prove that the bottom map is homotopic as maps of cubes to the map obtained by twice looping the loop-suspension map $X\to \Omega\Sigma X$ with $X=\Sigma^d\Omega(M_\bullet)_+$ and then show that the bottom map is sufficiently cartesian by a multirelative version of Freudenthal's suspension theorem. This strategy can indeed be implemented: to achieve the first step, one uses \cref{lem: left-right} and the homotopy that it provides, and for the second step one uses that if $X_\bullet$ is a strongly cocartesian $r$-cube of based spaces and if $k_i$ is the connectivity of the map $X_{\varnothing}\to X_{\{i\}}$, then the loop-suspension map of cubes $\Sigma^d \Omega(X_\bullet)_+\to \Omega \Sigma \Sigma^d \Omega(X_\bullet)_+
$ is $(2d-1+\sum_{i=1}^r(k_i-1))$-cartesian. This can be shown by an application of a more flexible version of the multirelative Blakers-Massey theorem \cite[Theorem 2.5]{GoodwillieCalculusII}.

\subsubsection{$\CE$ versus $\CbE$}
The main reason we used $\CbE(\ast,M)$ instead of $\CE(\ast,M)$ in Step \ref{step:partial} is the fact that the space $\CbE^{A_S}(\ast,M)$ is contractible for $\varnothing\neq S\subsetneq \ul{3}$, which allowed for a simple description of $\holim_{\varnothing\neq S\subseteq \ul{3}}\CbE^{A_S}(\ast,M)$, namely as $\Omega^2\CE^{\{t_1,t_2\}}(\ast,M)$ which is by \cref{lem:SOmega} equivalent to $\Omega^2(S^d\wedge \Omega(M)_+)$. The corresponding homotopy limit of $\CE^{A_S}(\ast,M)$ can be seen to be equivalent to the homotopy fibre of the inclusion $S^{T_\ast M}\to S^{T_\ast M}\wedge \Omega(M)_+$. This leads to a commutative diagram
\begin{equation}\label{equ:big-holim-diagram}\begin{tikzcd}[ar symbol/.style = {draw=none,"\textstyle#1" description,sloped},
  equivalence/.style = {ar symbol={\simeq}}, column sep=0.5cm]\CbE(\ast,M) \rar \dar & \CE(\ast,M) \rar \dar & \CI(\ast,M) \dar{\simeq} \\
\underset{\varnothing \neq S\subseteq \ul{3}}\holim\,\CbE^{A_S}(\ast,M) \rar \arrow[d,equivalence,start anchor={[xshift=-5pt]}] & \underset{\varnothing \neq S\subseteq \ul{3}}\holim\,\CE^{A_S}(\ast,M) \rar \arrow[d,equivalence,start anchor={[xshift=-5pt]}] & \underset{\varnothing \neq S\subseteq \ul{3}}\holim\,\CI^{A_S}(\ast,M) \arrow[d,equivalence,start anchor={[xshift=-5pt]}] \\[-0.4cm]
	\Omega^2 (S^{T_\ast M}\wedge \Omega(M)_+) & \Omega\hofib(S^{T_\ast M}\to S^{T_\ast M}\wedge \Omega(M)_+) & \Omega S^{T_\ast M} \end{tikzcd}\end{equation}
whose bottom row can, via the indicated equivalences, be identified with the evident homotopy fibre sequence relating the three spaces involved. In particular, since the inclusion $S^{T_\ast M}\to S^{T_\ast M}\wedge \Omega(M)_+$ has a left inverse, the map $\Omega\hofib(S^{T_\ast M}\to S^{T_\ast M}\wedge \Omega(M)_+)\ra\Omega S^{T_\ast M}$ is nullhomotopic, so the same holds for the map $\CE(\ast,M)\ra \CI(\ast,M)$.

The fact that this map is nullhomotopic actually holds more generally: the map $\CE(P,M)\to \CI(P,M)$ is nullhomotopic whenever the handle dimension of $\partial M \cap P \subset P$ is less than $d$, that is, whenever Smale-Hirsch theory applies. One proof of this fact goes by mapping $\CE(P,M)$ and $\CI(P,M)$ compatibly to spaces of sections relative to $P\cap\partial M$ of bundles over $P$ whose fibre over $\ast\in P$ are the spaces $\CE(\ast,M)$ and $\CI(\ast,M)$ respectively. By the discussion in \cref{sec:immersions} the map from $\CI(P,M)$ to the section space with fibres $\CI(\ast,M)$ is an equivalence, so to provide the claimed nullhomotopy, it suffices to produce a nullhomotopy of $\CE(\ast,M)\ra \CI(\ast,M)$ that varies continuously with $\ast\in P$. The nullhomotopy discussed below \eqref{equ:big-holim-diagram} has this property. When Smale-Hirsch theory does not apply, e.g.\,for concordance diffeomorphisms, this argument still shows that the map from $\CE(P,M)$ to the section space over $P$ with fibre $\CI(\ast,M)\simeq \Omega S^{T_\ast M}$ is nullhomotopic.

\begin{rem}This has an application to the map $\Diff_\partial(D^d ) \ra \Omega^{d} \SO(d)$ induced by taking derivatives. Namely, writing $D^d\cong D^{d-1}\times I$, this map fits into a commutative square
	\[\begin{tikzcd} \Diff_\partial(D^{d-1} \times I) \rar \dar{\inc} & \Omega^{d} \SO(d) \dar{\Omega\pr} \\
	\C(D^{d-1}) \rar & \Omega^{d} S^{d-1} \end{tikzcd}\]
whose bottom map is nullhomotopic by the discussion above, so we obtain a lift of the top map to a map of the form
	$\Diff_\partial(D^{d-1} \times I) \ra \Omega^{d} \SO(d-1)$.
\end{rem}

\subsubsection{Relation to ``Calculus I.''}
There is some overlap between the arguments in \cref{section:stability-point} above and those in the final section of \cite{GoodwillieCalculusI}. It is worth clarifying the situation.

In Section 3 of loc.cit.~Goodwillie in effect proved the $P=\ast$ case of the stability theorem \cref{thm:main}, by giving a description of the homotopy type of $\CE(\ast,M)$ and $\CE(\ast\times J,M\times J)$ in a range up to $2d-5$ (and more generally of $\CE(\ast\times D^p,M\times D^p)$). Let us recall how that went. The key in \cite[Section 3]{GoodwillieCalculusI} was a map
\begin{equation}\label{trace}
\CE(\ast,M)\lra \Omega^2Q(\Sigma^{d}\Omega(M)),
\end{equation}
which was shown to be $(2d-5)$-connected (see p.~19 and Lemma 3.16 loc.cit.). Note that the target receives a $(2d-4)$-connected map from $\Omega(\hofib(S^d\ra S^d\wedge \Omega(M)_+))$, so in this range, the map \eqref{trace} has the same form as the middle vertical map in \eqref{equ:big-holim-diagram}. Similarly, there are maps
\begin{equation}\label{trace disk}
	\CE(D^p,M)\lra \Omega^{2} Q(\Sigma^{d-p}\Omega(M))
\end{equation}
which are shown to be at least $(2d-2p-5)$-connected for $p\le d-3$ by induction on $p$, using the delooping trick (see the proof of Lemma 3.19 loc.cit.). Moreover, these maps are compatible with stabilisation in the evident sense, so in particular this implies (a) that we have a $(2d-5)$-connected map
\[\hocolim_p\,\CE(D^p,M\times D^p)\lra \Omega^{2} Q(\Sigma^{d}\Omega(M))\]
and (b) that the stabilisation map $\CE(\ast,M)\ra \CE(\ast\times J,M\times J)$ is $(2d-6)$-connected.

In \cite{GoodwillieCalculusI} these ideas were used to compute the first derivative of the stable concordance diffeomorphism functor in the ``homotopy calculus'' sense. The map \eqref{trace}, or the idea behind it, led to another map
\begin{equation}\label{global trace}
	\C(M)\lra \Omega^2 Q(\Lambda M/M),
\end{equation}
where $\Lambda M/M$ is the (homotopy) cofibre of the inclusion $M\ra \Lambda M=\Map(S^1,M)$ of $M$ into the free loop space of $M$ as the constant loops (see p.~24 of loc.cit.). It is also compatible with stabilisation, so that it gives a map
\[\hocolim_p\,\C(M\times D^p)\lra \Omega^2 Q(\Lambda M/M)\]
which eventually leads to a computation of the aforementioned first derivative. 

\begin{rem}\ 
\begin{enumerate}
\item In \cite{GoodwillieCalculusI} the maps \eqref{trace} and \eqref{global trace} were defined using a certain ``cobordism'' model for the targets (see p.\,19-20 and p.\,23-24 loc.cit.), whereas we constructed the related maps vertical maps in \eqref{equ:big-holim-diagram} by cutting holes in the interval $I$.
\item There was an oversight in \cite{GoodwillieCalculusI}: the compatibility of \eqref{trace disk} and \eqref{global trace}  with stabilisation was never explained. Goodwillie would like to repair that oversight by pointing out that the definition of these maps using the ``cobordism'' model is rather obviously compatible with stabilisation. In fact, if a family of concordance embeddings of $P$ in $M$, parametrised by the simplex $\Delta^k$, satisfies the transversality condition ``Hypothesis 3.18'' of loc.cit.~then the resulting family of concordance embeddings of $P\times J$ in $M\times J$ satisfies the same condition, and the resulting $k$-simplex in the cobordism space is essentially unchanged.
\item Although the cobordism approach was convenient for establishing compatibility with stabilisation (bypassing the need for the homotopy we constructed above using \cref{lem: left-right}) and was also adequate for establishing the absolute $(2d-5)$-connectedness of stabilisation in the point case, it may be difficult to use it to obtain the multirelative statement. In our proof, this multirelative statement in the $P=\ast$ case  is needed even for the absolute statement in the $P=D^p$ case.
\end{enumerate}
\end{rem}

\subsection{The case of a disc}\label{sec:proof-disc}
Returning to the proof of \cref{thm:main-relative}, we now give the argument in the case when $P$  is a $p$-disc $D^p\subset M$ with $D^p\cap \partial M=\partial D^p$, using the case $p=0$ from \cref{section:stability-point}. The assertion is that the stability $(r+1)$-cube 
\[\sCE(D^p,M_\bullet)=\big(\CE(D^p,M_\bullet)\xrightarrow{\sigma} \CE(D^p\times J,M _\bullet \times J)\big)\]
is $(2d-p-5+\Sigma)$-cartesian if $d-p\ge3$ and $d-q_i\ge3$ for all $i$. The proof is by double induction, using the induction hypothesis
\medskip
\begin{itemize}
	\item[$(\bfH_{p,k})$]For $D^p\subset M$ with $D^p\cap M=\partial D^p$, the cube $\sCE(D^p,M_\bullet)$ is $(k+d-p-3+\Sigma)$-cartesian.
\end{itemize}
\medskip
The goal is to prove $(\bfH_{p,d-2})$ for $d-p\ge 3$. From the case of a point considered in the previous subsection, we have $(\bfH_{0,d-2})$ for $d\ge 3$. Using \cref{lem:cubelemma} \ref{cubelemma:ii}, we also have $(\bfH_{p,0})$ for $d-p\ge3$, since both of the $r$-cubes $\CE(D^p,M_\bullet)$ and $\CE(D^p\times J,M_\bullet\times J)$ are $(d-p-2+\Sigma)$-cartesian by \cref{thm:multipl-disjunction}. The induction will be completed by showing that the statement $(\bfH_{p,k})$ follows from $(\bfH_{p,k-1})$ and $(\bfH_{p-1,k})$ if $d-p\ge3$. To do so, note that the scanning map from \cref{sec:delooping} is compatible with stabilisation, so that we have a scanning map of $(k+1)$-cubes
	\begin{equation}\label{equ:scanning-map-with-s} \sCE(D^p,M_\bullet) \xlra{\tau} \Omega s\CE(D^{p-1},M_\bullet).\end{equation}
By $(\bfH_{p-1,k})$ the cube $s\CE(D^{p-1},M_\bullet)$ is $(k+d-(p-1)-3+\Sigma)$-cartesian, so the target of this map is $(k+d-p-3+\Sigma)$-cartesian. To show $(\bfH_{p,k})$, i.e.\,that the source of \eqref{equ:scanning-map-with-s} is also $(k+d-p-3+\Sigma)$-cartesian, it suffices by \cref{lem:cubelemma} \ref{cubelemma:i} to show that the entire $(k+2)$-cube \eqref{equ:scanning-map-with-s} is $(k+d-p-3+\Sigma)$-cartesian. To this end, note that our identification of the scanning map with (\ref{eqref:inclusion-delooping}) is compatible with stabilisation: the cube \eqref{equ:scanning-map-with-s} is equivalent to the inclusion 
\[\sCE\big(D^p_{[\epsilon,1]},(M\backslash T)_\bullet\backslash D_{[-1,-\epsilon]}^p)\overset{\subset}\lra \sCE\big(D^p_{[\epsilon,1]},(M\backslash T))\]
of $(r+1)$-cubes. Statement $(\bfH_{p,k-1})$ shows that this $(r+2)$-cube is $(k-1+d-p-3+\Sigma+(d-p-2))$-cartesian. (Here one sets $\smash{Q_{r+1}\coloneq D_{[-1,-\epsilon]}^p}$.) Therefore \eqref{equ:scanning-map-with-s} is $(k-1+d-p-3+\Sigma+(d-p-2))$-cartesian, so in particular $(k+d-p-3+\Sigma)$-cartesian, since $d-p-3\ge 0$. 

\subsection{The general case}\label{sec:proof-general}
Finally, we establish the general case of \cref{thm:main-relative}: that the stability $(k+1)$-cube  $\sCE(P,M_\bullet)$ is $(2d-p-5+\Sigma)$-cartesian if $d-p\ge3$ and $d-q_i\ge3$ for all $i$. Recall that $p$ was defined as the handle dimension of the inclusion $\partial M\cap P\subset P$, so there exists a handle decomposition of $P$ relative to a closed collar on $\partial M\cap P$ with handles of index at most $d-3$. The proof will be an induction over the number of handles $k$ of such a decomposition. 

If $k=0$, then $P$ is a closed collar on $P\cap\partial M$, so $\sCE(P,M\backslash Q_\bullet)$ is objectwise contractible by \cref{lem:collar} and thus in particular $\infty$-cartesian by \cref{cor:cube-of-cubes} \ref{enum:constant-cubes}. To go from $k-1$ to $k$, suppose that $P=P'\cup H$ where $H=D^i\times D^{\dim(P)-i}$ is a handle of index $i\le p$ disjoint from $\partial M\cap P$, and assume that the claim holds for $P'$. The base of the restriction map $\sCE(P,M\backslash Q_\bullet)\ra \sCE(P',M\backslash Q_\bullet)$ satisfies the cartesianness bound of the claim by the induction hypothesis, so using \cref{lem:cubelemma} \ref{cubelemma:i} and \ref{cubelemma:iii} it suffices that the $(k+1)$-cubes of homotopy fibres over all basepoints satisfy the cartesianness bound of the claim. In fact, since the $\CE(P',M_\varnothing)$ is connected by \cref{thm:hudson}, it suffices to take homotopy fibres of the inclusion. As a result of the parameterised isotopy extension theorem, this $(k+1)$-cube of homotopy fibres over the inclusion is equivalent to $\sCE(H,M'_\bullet)$ where $M'\subset M$ is the complement of an open tubular neighbourhood of $P'$ disjoint from the $Q_i$'s and is so that $H\cap\partial M'=(\partial D^i)\times D^{\dim(P)-i}$. Now note that $H$ is a closed disc bundle as in \cref{lem:disc-bundle}, so the restriction map $\sCE(H,M'_\bullet)\ra \sCE(D^i,M'_\bullet)$ along $D^i\times \{0\}\subset D^i\times D^{\dim(P)-i}=H$ is a objectwise equivalence. But $\sCE(D^i,M'_\bullet)$ satisfies the cartesianness bound of the claim by the case treated in the previous subsection, so the proof is complete.

\section{Applications to concordance diffeomorphisms and homeomorphisms}\label{sec:applications}In this section we explain some applications of \cref{thm:main} to spaces of concordance diffeomorphisms (see \cref{sec:smooth-conc}) and of concordance homeomorphisms (see \cref{sec:topological-conc}).

\subsection{Smooth concordances} \label{sec:smooth-conc}Recall from the introduction that $\C(M)$ for a compact smooth $d$-manifold is the topological group of diffeomorphisms of $M \times I \to M \times I$ that agree with identity on a neighbourhood of $M \times \{0\} \cup \partial M \times I$, equipped with the smooth topology. Note that $\C(M)=\CE(M,M)$ since $M$ is compact. We begin with an invariance result for $\C(M\times J)/\C(M)$ with respect to attaching certain handles.

\begin{thm}\label{thm:handle-invariance}Let $M$ and $N$ be compact $d$-manifolds such that $N$ is obtained from $M$ by attaching finitely many handles of index $\ge k$. If $k\ge3$ then the map
\[\C(M\times J)/\C(M)\lra \C(N\times J)/\C(N)\]
 is $(d+k-5)$-connected.
\end{thm}

\begin{proof}The map in consideration is the induced map on vertical homotopy fibres of
	\[\begin{tikzcd}
	\BC(M)\rar\dar &\dar\BC(N)\dar\\
	\BC(M\times J)\rar&\BC(N\times  J)
	\end{tikzcd}\]
so we may equivalently show that the map on horizontal homotopy fibres has the claimed connectivity. By induction over the number of handles, it suffices to prove the case $N=M\cup H$ where $H$ is a single $k$-handle with $k\ge3$. In this case, as a result of the parametrised isotopy extension theorem together with \cref{lem:disc-bundle} and \cref{thm:hudson}, the map on horizontal homotopy fibres agrees up to equivalence with $\CE(D^{d-k},M\cup H)\ra \CE(D^{d-k}\times J,(M\cup H)\times  J)$ where $D^{d-k}\subset H$ is a cocore of the $k$-handle. The latter map is $(d+k-5)$-connected by \cref{thm:main}, so the claim follows.
\end{proof}

\begin{cor}Let $M\hookrightarrow N$ be a $k$-connected embedding between compact $d$-manifolds. If
\begin{enumerate}
\item\label{enum:ctd-i}  the inclusion $\partial M\subset M$ induces an equivalence of fundamental groupoids and
\item\label{enum:ctd-ii} $2\le k\le d-4$,
\end{enumerate}
then the map
\[\C(M\times J)/\C(M)\lra \C(N\times J)/\C(N)\]
is $(d+k-4)$-connected.
\end{cor}
\begin{proof}
We may assume without loss of generality that $e$ lands in the interior of $N$, so that the complement gives a bordism $W\coloneq N\backslash \interior(M)\colon \partial M\leadsto \partial N$. By \cref{thm:handle-invariance}, it suffices to show that $W$ can be built from a closed collar on $\partial M$ by attaching handles of index $>k$. This follows from handle trading  (see e.g.\,\cite[Theorem 3]{WallConnectivity}) using that the assumptions \ref{enum:ctd-i} and \ref{enum:ctd-ii} imply that the inclusion $\partial M\subset W$ is $k$-connected.
\end{proof}

The next result involves the notion of a tangential $k$-type which we briefly recall. Two $d$-manifolds $M$ and $N$ are said to have the \emph{same tangential $k$-type} for some $k\ge0$ if there exists a space $B$ and factorisations $M\ra B\ra \BO$ and $N\ra B\ra\BO$ of the classifier of the respective stable tangent bundles such that the maps $M\ra B$ and $N\ra B$ are $k$-connected. A codimension $0$ embedding $M\hookrightarrow N$ is an \emph{equivalence on tangential $k$-types} if there exists a space $B$ and a factorisation $N\ra B\ra \BO$ such that the map $N\ra B$ and the composition $M\hookrightarrow N\ra B$ are $k$-connected.

\begin{ex}\ 
\begin{enumerate}
\item Two spin $d$-manifolds have the same tangential $2$-types if their fundamental groupoids are equivalent. An embedding between such manifolds is an equivalence on tangential $2$-types if it induces an equivalence on fundamental groupoids.
\item Any $k$-connected codimension $0$ embedding is an equivalence on tangential $k$-types.
\end{enumerate}
\end{ex}

\begin{thm}\label{thm:tangential}Let $M$ and $N$ be $d$-manifolds and $k$ an integer with $2\le k< d/2$.
\begin{enumerate}
\item\label{enum:i-tangential} If $M$ and $N$ have the same tangential $k$-type, then there exists an equivalence
\[\tau_{\le d+k-5}\big(\C(M\times J)/\C(M)\big)\simeq \tau_{\le d+k-5}\big(\C(N\times J)/\C(N)\big)\]
between the indicated Postnikov truncations.
\item\label{enum:ii-tangential} Assume in addition $k<(d-1)/2$. For any codimension $0$ embedding $M\hookrightarrow N$ that is an equivalence on tangential $k$-types, the induced map
\[(\C(M\times J)/\C(M))\lra (\C(N\times J)/\C(N))\]
is an equivalence on Postnikov ({d+k-5})-truncations.
\end{enumerate}
\end{thm}

\begin{proof}[Proofs of \cref{thm:tangential} and \cref{cor:relative-concordance}]
\cref{thm:tangential} is a direct consequence of \cite[Theorem 5.7]{KKDisc} applied to $\tau_{\le k+d-5}(\C(-\times J)/\C(-))$, considered as a functor from the category of compact $d$-manifolds and isotopy classes of codimension zero embeddings to the homotopy category of spaces. The assumption of the theorem is satisfied by \cref{thm:handle-invariance}. 

\cref{cor:relative-concordance} almost follows from the case $M=D^d$ of \cref{thm:tangential} \ref{enum:ii-tangential} except that we do not require the additional assumption $k<(d-1)/2$. That this assumption is not necessary in the special case $M=D^d$ follows from the second part of \cite[Theorem 5.7]{KKDisc}\end{proof}

\subsection{Topological concordances}\label{sec:topological-conc}
We write $\C^{\TOP}(M)$ for the topological group of \emph{topological} concordances by which we mean the space of homeomorphisms of $M\times I$ that are the identity in a neighbourhood of $M\times\{0\}\cup\partial M\times I$, equipped with the compact-open topology. The definition of the stabilisation map makes equal sense for topological concordances.

\begin{prop}\label{prop:smooth-vs-topological-stability}Let $M$ be a compact smooth $d$-manifold with $d\ge5$. If 
\[\BC(D^d)\lra\BC(D^{d}\times J)\]
is $k$-connected for some $k\ge0$, then for any smooth compact $d$-manifold $M$ the square
	\[\begin{tikzcd}
	\BC(M)\rar\dar&\BC^{\TOP}(M)\dar\\
	\BC(M\times J)\rar&\BC^{\TOP}(M\times J)
	\end{tikzcd}\]
is $k$-cartesian. The same implication holds rationally or $p$-locally for any prime $p$.
\end{prop}

\begin{proof}
As explained in \cite[p.\,453--458]{BurgheleaLashofStability}, it follows from smoothing theory that there is a map of homotopy fibre sequences
	\begin{equation}\label{equ:smoothing-theory}
	\begin{tikzcd}
	\C(M)\rar\dar&\C^{\TOP}(M)\dar\rar&\Sect_\partial(F_d \times_{\oO(d)} \Fr(M) \to M)\dar\\
	\C(M\times J)\rar&\C^{\TOP}(M\times J)\rar&\Sect_\partial(\Omega F_{d+1} \times_{\oO(d)} \Fr(M) \to M)
	\end{tikzcd}
	\end{equation}
where the right terms are the space of sections, fixed on the boundary, of the indicated bundles where $F_d\coloneq\hofib(\TOP(d)/\oO(d)\ra \TOP(d+1)/\oO(d+1))$ is the homotopy fibre of the  map induced by taking products with the real line. The rightmost vertical map is induced by the stabilisation map $F_d\ra \Omega F_{d+1}$ \cite[p.\,450]{BurgheleaLashofStability}, so its homotopy fibre is given by a similar space of sections $\Sect_\partial(G_d \times_{\oO(d)} \Fr(M) \to M)$ of a bundle over $M$ whose fibre is the space $G_d\coloneq\hofib(F_d\ra \Omega F_{d+1})$. 
 
For $M = D^d$ the middle terms in \eqref{equ:smoothing-theory} are contractible by the Alexander trick, so since $\C(D^d)\ra\C(D^{d}\times J)$ is $(k-1)$-connected by assumption it follows that the section space $\Omega \Sect_\partial(G_d \times_{\oO(d)} \Fr(D^d) \to D^d) \simeq \Omega^{d+1}G_d$ is $(k-2)$-connected. Moreover, as $F_d$ is $(d+1)$-connected for $d \geq 5$ by \cite[Essay V.5.2]{KirbySiebenmann}, the space $G_d$ is $d$-connected and thus in fact $(d+k-1)$-connected. For a general $M$, this implies that the right map in \eqref{equ:smoothing-theory} is $k$-connected by obstruction theory, so the left square is $(k-1)$-cartesian and the claim follows.
 
The rational (or $p$-local) addendum follows by the same argument: the case of a disc shows that $G_d$ is $d$-connected and rationally (or $p$-locally) $(d+k-1)$-connected, so the claim follows again from obstruction theory.\end{proof}

\begin{cor}For a smoothable $1$-connected compact spin $d$-manifold $M$ with $d\ge6$,
	\[\BC^\Top(M) \lra \BC^\Top(M \times J)\]
is rationally $\min(d-4,\lfloor\tfrac{3}{2}d\rfloor-9)$-connected.
\end{cor}

\begin{proof}This follows from \cref{prop:smooth-vs-topological-stability} since the maps $\BC(D^d)\ra \BC(D^{d+1})$ and $\BC(M) \to \BC(M \times J)$ are rationally $\min(d-3,\lfloor\tfrac{3}{2}d\rfloor -8)$-connected by \cref{cor:rational-stable-range} for $k=2$.
\end{proof}

\begin{rem}As the forgetful map $\C^\PL(M) \to \C^\Top(M)$ from the PL-version of the space of concordance homeomorphisms is an equivalence for all PL-manifolds $M$ of dimension $\geq 5$ \cite[Theorem 6.2]{BurgheleaLashofII}, the results of this section also apply to $\C^\PL(M)$.\end{rem}

\bibliographystyle{amsalpha}
\bibliography{literature}

\providecommand{\bysame}{\leavevmode\hbox to3em{\hrulefill}\thinspace}
\providecommand{\MR}{\relax\ifhmode\unskip\space\fi MR }
\providecommand{\MRhref}[2]{%
  \href{http://www.ams.org/mathscinet-getitem?mr=#1}{#2}
}
\providecommand{\href}[2]{#2}
\begin{thebibliography}{KRW21}

\bibitem[BL74]{BurgheleaLashofII}
D.~Burghelea and R.~Lashof, \emph{The homotopy type of the space of
  diffeomorphisms. {II}}, Trans. Amer. Math. Soc. \textbf{196} (1974), 37--50.
  \MR{356103}

\bibitem[BL77]{BurgheleaLashofStability}
\bysame, \emph{Stability of concordances and the suspension homomorphism}, Ann.
  of Math. (2) \textbf{105} (1977), no.~3, 449--472.

\bibitem[BLR75]{BurgheleaLashofRothenberg}
D.~Burghelea, R.~Lashof, and M.~Rothenberg, \emph{Groups of automorphisms of
  manifolds}, Lecture Notes in Mathematics, Vol. 473, Springer-Verlag,
  Berlin-New York, 1975.

\bibitem[Cer70]{Cerf}
J.~Cerf, \emph{La stratification naturelle des espaces de fonctions
  diff\'{e}rentiables r\'{e}elles et le th\'{e}or\`eme de la pseudo-isotopie},
  Inst. Hautes \'{E}tudes Sci. Publ. Math. (1970), no.~39, 5--173.

\bibitem[GK15]{GoodwillieKlein}
T.~G. Goodwillie and J.~R. Klein, \emph{Multiple disjunction for spaces of
  smooth embeddings}, J. Topol. \textbf{8} (2015), no.~3, 651--674.

\bibitem[Goo82]{GoodwillieThesis}
T.~G. Goodwillie, \emph{A multiple disjunction lemma for smooth concordance
  embeddings}, 1982, Thesis (Ph.D.)--Princeton University.

\bibitem[Goo90a]{GoodwillieCalculusI}
\bysame, \emph{Calculus. {I}. {T}he first derivative of pseudoisotopy theory},
  $K$-Theory \textbf{4} (1990), no.~1, 1--27.

\bibitem[Goo90b]{Goodwillie}
\bysame, \emph{A multiple disjunction lemma for smooth concordance embeddings},
  Mem. Amer. Math. Soc. \textbf{86} (1990), no.~431, viii+317.

\bibitem[Goo03]{GoodwillieCalculusIII}
\bysame, \emph{Calculus. {III}. {T}aylor series}, Geom. Topol. \textbf{7}
  (2003), 645--711. \MR{2026544}

\bibitem[Goo92]{GoodwillieCalculusII}
\bysame, \emph{Calculus. {II}. {A}nalytic functors}, $K$-Theory \textbf{5}
  (1991/92), no.~4, 295--332.

\bibitem[GW99]{GoodwillieWeiss}
T.~G. Goodwillie and M.~Weiss, \emph{Embeddings from the point of view of
  immersion theory. {II}}, Geom. Topol. \textbf{3} (1999), 103--118.
  \MR{1694808}

\bibitem[Hat75]{HatcherSimple}
A.~E. Hatcher, \emph{Higher simple homotopy theory}, Ann. of Math. (2)
  \textbf{102} (1975), no.~1, 101--137. \MR{383424}

\bibitem[Hat78]{HatcherSurvey}
\bysame, \emph{Concordance spaces, higher simple-homotopy theory, and
  applications}, Algebraic and geometric topology ({P}roc. {S}ympos. {P}ure
  {M}ath., {S}tanford {U}niv., {S}tanford, {C}alif., 1976), {P}art 1, Proc.
  Sympos. Pure Math., XXXII, Amer. Math. Soc., Providence, R.I., 1978,
  pp.~3--21. \MR{520490}

\bibitem[Hud70]{HudsonConcordance}
J.~F.~P. Hudson, \emph{Concordance, isotopy, and diffeotopy}, Ann. of Math. (2)
  \textbf{91} (1970), 425--448. \MR{259920}

\bibitem[Igu88]{IgusaStability}
K.~Igusa, \emph{The stability theorem for smooth pseudoisotopies}, $K$-Theory
  \textbf{2} (1988), no.~1-2, vi+355.

\bibitem[KK22]{KKDisc}
M.~Krannich and A.~Kupers, \emph{{The Disc-structure space}}, arXiv:2205.01755.

\bibitem[KRW21]{KRWorthogonal}
M.~Krannich and O.~Randal-Williams, \emph{{Diffeomorphisms of discs and the
  second Weiss derivative of BTop(-)}}, arXiv:2109.03500.

\bibitem[KS77]{KirbySiebenmann}
R.~C. Kirby and L.~C. Siebenmann, \emph{Foundational essays on topological
  manifolds, smoothings, and triangulations}, Princeton University Press,
  Princeton, N.J.; University of Tokyo Press, Tokyo, 1977, With notes by John
  Milnor and Michael Atiyah, Annals of Mathematics Studies, No. 88.

\bibitem[Men93]{Meng}
G.~Meng, \emph{The stability theorem for smooth concordance imbeddings}, 1993,
  Thesis (Ph.D.)--Brown University.

\bibitem[Wal71]{WallConnectivity}
C.~T.~C. Wall, \emph{Geometrical connectivity. {I}}, J. London Math. Soc. (2)
  \textbf{3} (1971), 597--604. \MR{290387}

\bibitem[Wat09]{Watanabe}
T.~Watanabe, \emph{On {K}ontsevich's characteristic classes for higher
  dimensional sphere bundles. {I}. {T}he simplest class}, Math. Z. \textbf{262}
  (2009), no.~3, 683--712. \MR{2506314}

\bibitem[Wei99]{WeissImmersion}
M.~Weiss, \emph{Embeddings from the point of view of immersion theory. {I}},
  Geom. Topol. \textbf{3} (1999), 67--101. \MR{1694812}

\bibitem[Wei11]{WeissImmersionErrata}
\bysame, \emph{Erratum to the article {E}mbeddings from the point of view of
  immersion theory: {P}art {I}}, Geom. Topol. \textbf{15} (2011), no.~1,
  407--409. \MR{2776849}

\bibitem[WW01]{WWsurvey}
M.~Weiss and B.~Williams, \emph{Automorphisms of manifolds}, Surveys on surgery
  theory, {V}ol. 2, Ann. of Math. Stud., vol. 149, Princeton Univ. Press,
  Princeton, NJ, 2001, pp.~165--220. \MR{1818774}

\end{thebibliography}

\bigskip

\end{document}